\documentclass[10.5pt]{article}
\usepackage{mathtools}

\usepackage{amsmath,amssymb,amstext,amsfonts}
\usepackage{amsthm}
\usepackage{enumerate}
\usepackage{tcolorbox}
\usepackage{lipsum}
\usepackage{amsfonts}
\usepackage{graphicx}
\usepackage{epstopdf}
\usepackage{algorithm}
\usepackage{algorithmic}
\usepackage{booktabs}
\usepackage[numbers]{natbib}

\usepackage{subcaption}
\usepackage{float}
\usepackage[margin=0.9in]{geometry}
\usepackage{listings}
\usepackage{xcolor}
\usepackage{subcaption} 

\usepackage[framemethod=TikZ]{mdframed}

\usepackage{algorithm}
\usepackage{algorithmic}

\mathtoolsset{showonlyrefs=true}
\newtheorem{Problem}{Problem}

\usepackage{graphicx}
\usepackage[utf8]{inputenc}
\newtheorem{theorem}{Theorem}[section]
\newtheorem{lemma}[theorem]{Lemma}
\newtheorem{proposition}[theorem]{Proposition}

\newtheorem{assumption}{Assumption}
\usepackage{systeme}
\usepackage{comment}
\usepackage{dsfont}
\newtheorem*{remark}{Remark}

\tcbuselibrary{theorems}
\theoremstyle{definition}

\newcounter{texercise}
\newwrite\solout
\def\openoutsol{\immediate\openout\solout\jobname.sol}  \def\writesol#1{\immediate\write\solout{\noexpand\processsol{\thetexercise}{#1}}}

\newcounter{mytheorem}[section]

\numberwithin{equation}{section}
\usepackage[english]{babel}
\usepackage[bookmarks]{hyperref}

\usepackage{hyperref}
\hypersetup{
    linkcolor=blue,
    filecolor=magenta,      
    urlcolor=cyan,
    pdfpagemode=FullScreen,
}
\DeclareMathOperator{\Tr}{Tr}

\newcommand{\bR}{\mathbb{R}}

\newcommand{\E}{\mathbb{E}}
\newcommand{\mcal}[1]{\mathcal{#1}}

\begin{document}
\title{A Sample-Wise Adjoint Regression Framework for Mean-Field Control with Connections to Adjoint Matching}
\author{
Hui Sun \thanks{Department of Financial and Actuarial Mathematics, 
School of Mathematics \& Physics. Xi’an Jiaotong-Liverpool University, Suzhou 215123, P.R.China. \textbf{Funding:} Research Development Fund (RDF) Xi’an Jiaotong-Liverpool University (XJTLU), No.: RDF-24-02-029.  \tt{Email:hui.sun@xjtlu.edu.cn}}
}
\date{This Version: \today }
\maketitle

\begin{abstract}
This work proposes a novel numerical approach for solving mean-field control (MFC) problems using an adjoint-based optimization framework motivated by the stochastic maximum principle (SMP). Rather than solving the adjoint processes $(Y_t,Z_t)$ explicitly, we construct sample-wise unbiased estimators of their discretized counterparts and use them to approximate the Hamiltonian control gradient through recursive regression. The control is then updated by a gradient-descent scheme. This approach differs from conventional deep-learning methods, which typically follow a ``discretize-then-optimize'' paradigm and directly optimize a globally parameterized control through the discretized objective. Numerical experiments demonstrate competitive, and in many cases improved performance compared with direct deep-learning approaches. The sample-wise and regression-based structure also supports scalable implementation in high-dimensional settings and makes the method attractive for generative modeling problems involving particle-based representations and distributional objectives. On the theoretical side, we establish a connection between a particular class of MFC problems and the \emph{adjoint matching} framework. Using the SMP under the mean field control setting, we further show that the self-consistent critical point of the resulting mean-field adjoint matching loss coincides with the optimal control.
\end{abstract}



\section{Introduction}
Stochastic optimal control problems (SOCPs) have found broad applications across science and engineering \cite{Hui1, Bensoussan1_T, carmona1, Carmona2, jiequn2, Du1, Zhao1, jiequn1, Hanson, Peng1, Peng4, pham1, pham2, pham3, Hui2}. In recent years, the subject has increasingly attracted attention, particularly with the emergence of machine learning and deep learning assisted computational methods. A substantial literature has been developed on both the theoretical foundations of stochastic optimal control \cite{Bensoussan1_T, Yong1, Peng1, Haussmann, pham1} and its numerical approximation \cite{Zhao1, Du1, Peng3, Hanson, Hui2, Kushner}. Among the early deep-learning approaches, Han et al.~\cite{jiequn1} parameterized the control function by a deep neural network and trained the network by minimizing the underlying cost functional. Related ideas have subsequently been applied to mean-field control and mean-field games in \cite{carmona1, Carmona2, Carmona3}. In these approaches, the distribution $\mcal{L}_X$ appearing in the coefficients of the McKean--Vlasov SDE is approximated by the empirical distribution of an interacting particle system, motivated by propagation-of-chaos arguments, while the resulting distributional interactions are approximated through ensemble averages. This particle based formulation integrates naturally with standard deep learning architectures and optimization procedures. At the same time, the close connection between stochastic optimal control and Hamilton--Jacobi--Bellman (HJB) equations provides an alternative route for the numerical treatment of high dimensional PDEs. Analogously, mean-field control is related to Bellman equations posed on the Wasserstein space of probability measures \cite{pham4}. These control-theoretic structures have been exploited for solving high-dimensional PDEs; see, for example, \cite{Cai,Lars}. Numerical methods for PDEs and value functions involving distributional dependence have also been developed in, among others, \cite{Teichmann,pham3}.

These approaches largely fall within a ``discretize-then-optimize'' framework: the control problem is first discretized, after which the numerical solver used to simulate the controlled SDE is differentiated through during optimization. Although conceptually straightforward and successful in many applications, this strategy can be memory intensive, since intermediate solver states and computational graphs must be retained or recomputed, often requiring techniques such as gradient checkpointing. Moreover, the resulting training procedure is largely black-box in nature and does not explicitly exploit much of the classical analytical structure of the underlying control problem. In contrast, the stochastic maximum principle (SMP) provides a systematic approach to SOCP and MFC; see, for example, Section~4.4 of \cite{Carmona0} and \cite{Anderson}. From a numerical perspective, the SMP is naturally ``optimize-then-discretize'' in spirit: the first-order optimality system is derived at the continuous level before being approximated numerically.

More specifically, the stochastic maximum principle (SMP) \cite{Peng4,Carmona0} characterizes the first variation of the cost functional through an adjoint system, which in the mean-field setting is typically represented by a mean-field backward stochastic differential equation (mean-field BSDE). Under suitable regularity and convexity conditions, the optimal control can be characterized through minimization of the associated Hamiltonian. Even when a closed-form optimizer is not available, the resulting directional derivative of the cost functional still provides sufficient information for constructing gradient-based numerical schemes for the underlying stochastic optimal control problem.

For the classical stochastic optimal control problem, \cite{Zhao1} develops a projected gradient-descent method in which the discretized adjoint processes $(Y_t,Z_t)$ are solved at each iteration and subsequently used to approximate the control gradient. Such an approach, however, relies heavily on spatial discretization and mesh-based approximation, and is therefore subject to the curse of dimensionality. The difficulty becomes even more pronounced in mean-field control, where the coefficients and adjoint equations may additionally depend on the distribution of the state. In \cite{Hui1,Hui3}, it is observed that the adjoint processes need not be reconstructed explicitly as functions of the state. Instead, sample-wise or particle approximations of the adjoint variables can be combined with stochastic gradient descent to obtain efficient numerical schemes in higher dimensions. A limitation of these approaches, however, is that the admissible controls are restricted to deterministic functions of time, whereas the optimal control in many stochastic control problems is naturally of feedback form.

In light of the discussion above, in this work, we design a new maximum based adjoint approach for solving MFC where the backbone is the sample-wise approximation of the discretized mean-field BSDE. More specifically, we construct and discretize the adjoint mean-field forward-backward SDEs to obtain $(Y^N_{t_n}, Z^N_{t_n})$. Instead of solving this decoupled forward-backward system explicitly, we first design a backward sample-wise Monte Carlo type scheme whose conditional expectation equals the $(Y^N_{t_n}, Z^N_{t_n})$, then we use a propagation of Chaos type approximation to achieve a full discretized approximation to $(Y^N_{t_n}, Z^N_{t_n})$. Importantly, no training of neural network is needed as the goal is to only obtain approximately unbiased estimate of $(Y^N_{t_n}, Z^N_{t_n})$. As such, the sample-wise approximation of the gradient loss objectives are achieved and its functional form is obtained via regression. Finally, the objective control is updated via the descent algorithm in the functional form. 

On the theoretical side, we identify a direct connection between the sample-wise adjoint approximation developed in this work and the \emph{Adjoint Matching} framework of \cite{DomingoEnrich1}. The stochastic optimal control problem considered there is recovered as a special case of our MFC formulation when the distributional dependence is removed and the dynamics are restricted to a control-affine drift with quadratic control cost and state-independent diffusion. In this case, the adjoint equation arising from the stochastic maximum principle coincides with the \emph{lean} pathwise adjoint ODE of \cite{DomingoEnrich1}, since the control is treated as an independent argument of the Hamiltonian and derivatives of the feedback map do not enter the adjoint equation. Its time discretization therefore has the same form as the sample-wise approximation of the discretized adjoint process $Y^N_{\cdot}$ used in our numerical scheme.

Moreover, the lean Adjoint Matching formulation in \cite{DomingoEnrich1} reduces stochastic optimal control to a recursive regression problem and is shown to have the optimal control as its self-consistent critical point. We extend this connection to the mean-field setting, where additional Lions-derivative terms arise from the dependence on the state distribution. We derive the corresponding mean-field lean adjoint matching formulation and relate it to the sample-wise adjoint approximation in our general MFC scheme. Using the SMP under the MFC setting with sufficient regularity assumptions, we show that the self-consistent critical point of the mean-field adjoint matching loss coincides with the optimal mean-field control.

The formulation of our work could be useful and meaningful for practical applications in data science. For instance, the MFC formulation considered in this work also naturally encompasses the soft-constrained Schr{\"o}dinger bridge problem studied in \cite{MaTanXu1}. In its stochastic control formulation, the classical Schr{\"o}dinger bridge problem can be viewed as a control-affine diffusion with a quadratic control cost, subject to the hard terminal constraint that the law of $X_T$ coincides with a prescribed target distribution. The soft-constrained formulation relaxes this requirement by allowing the terminal law to deviate from the target distribution and introducing a penalty term that measures this discrepancy in the objective functional. As a result, the constrained Schr{\"o}dinger bridge problem is approximated by a family of unconstrained McKean--Vlasov control problems, where the mean-field dependence enters through the terminal distributional cost. Under suitable assumptions, \cite{MaTanXu1} establishes the existence of optimal controls for the penalized problems and shows that, as the penalty parameter increases, the corresponding optimal controls and value functions converge to those of the classical Schr{\"o}dinger bridge problem, with a linear convergence rate. This connection suggests that the numerical MFC framework developed here may also provide a natural approach for soft-constrained Schr{\"o}dinger bridge problems arising in generative modeling. In addition, as in \cite{pham5}, the running cost may incorporate a discrepancy term between the law $\mcal{L}_{X_t}$ of the controlled state and a prescribed family of intermediate target distributions $\rho_t$. This allows the generative task to enforce not only terminal distribution matching but also intermediate law constraints along the entire time horizon. Compared with \cite{pham5}, our formulation is developed within a more general MFC framework and adopts a different numerical strategy rather than relying on neural residual solvers of the type introduced in \cite{HanJentzen1} to approximate the associated McKean--Vlasov FBSDE system, we optimize the control through sample-wise adjoint estimation and regression-based updates. This leads to a comparatively direct implementation while retaining the ability to handle both terminal and intermediate distributional objectives. 

As a further contribution, the proposed framework naturally extends to control of the diffusion coefficient. Consider
\begin{align}
    dX_t=b(t,X_t,\mu_t,u_t)\,dt+\sigma(t,X_t,\mu_t,u_t)\,dW_t,
\end{align}
with Hamiltonian $ H(t,x,\mu,y,z,u)=b(t,x,\mu,u)^Ty+\text{tr}(\sigma(t,x,\mu,u)^Tz)+f(t,x,\mu,u).$ Then
\begin{align}
    \partial_u H
    =(\partial_u b)^Ty+\text{tr}(\partial_u\sigma^Tz)+\partial_u f.
\end{align}
Hence, once the diffusion coefficient is controlled, the adjoint component $Z_t$ enters the control gradient directly. Since our numerical scheme already provides sample-wise approximations of the discretized $Z$ process, the same regression-based optimization procedure can be used to update both the drift and diffusion controls. This gives an additional degree of flexibility beyond the standard control-affine drift formulation and is potentially useful in generative modeling, where both transport and stochasticity influence the evolution of the target distribution. While the correctness of diffusion control is demonstrated in the classical linear quadratic MFC examples presented later, its potential benefits for generative modeling remain to be investigated systematically in future work.

\section{The Mean Field Control problem}
\subsection{Notation and background}
In this section, we introduce the notation and background material needed for the subsequent algorithm design and analysis. We denote by $\mcal{P}_2(\bR^d)$ the space of probability measures on $\bR^d$ with finite second moment, i.e.,
\begin{align}
    \mcal{P}_2(\bR^d):=\left\{\mu\in\mcal{P}(\bR^d):\int_{\bR^d}|x|^2\mu(dx)<\infty\right\}.
\end{align}
For $\mu,\nu\in\mcal{P}_2(\bR^d)$, let $\Pi(\mu,\nu)$ denote the set of couplings of $\mu$ and $\nu$, namely,
\begin{align}
    \Pi(\mu,\nu):=\left\{\pi\in\mcal{P}(\bR^d\times\bR^d):\pi(A\times\bR^d)=\mu(A),\ \pi(\bR^d\times B)=\nu(B)\right\},
\end{align}
for all Borel sets $A,B\subset\bR^d$. The $2$-Wasserstein distance between $\mu$ and $\nu$ is then defined by
\begin{align}
    W_2^2(\mu,\nu):=\inf_{\pi\in\Pi(\mu,\nu)}\int_{\bR^d\times\bR^d}|x-y|^2\pi(dx,dy).
\end{align}

In section 2.4, we will perform analysis on the space of probability measures, for which a notion of differentiation with respect to a probability measure is required. We use the Lions derivative, introduced by P.-L. Lions through a lifting argument on a Hilbert space. Let
\begin{align}
    u:\mcal{P}_2(\bR^d)\rightarrow\bR,
\end{align}
and consider its lift
\begin{align}
    U:L^2(\Omega;\bR^d)\rightarrow\bR,\qquad U(X):=u(\mcal{L}_X),
\end{align}
where $\mcal{L}_X$ denotes the law of the random variable $X$. We say that $u$ is Lions differentiable at $\mu\in\mcal{P}_2(\bR^d)$ if, for some $X\in L^2(\Omega;\bR^d)$ satisfying $\mcal{L}_X=\mu$, the lifted function $U$ is Fr\'echet differentiable at $X$. By the Riesz representation theorem, we identify the Fr\'echet derivative $DU(X)$ with an element of $L^2(\Omega;\bR^d)$ such that
\begin{align}
    DU(X)[Y]=\E\left[DU(X)\cdot Y\right],\qquad Y\in L^2(\Omega;\bR^d).
\end{align}
The derivative admits a representation of the form
\begin{align}
    DU(X)=D_\mu u(\mcal{L}_X)(X),
\end{align}
where $D_\mu u(\mu):\bR^d\rightarrow\bR^d$ is a deterministic measurable function defined $\mu$-a.e.. We refer to $D_\mu u(\mu)(x)$ as the Lions derivative of $u$ at $\mu$, evaluated at $x$. If a version of $D_\mu u(\mu)(x)$ can be chosen such that the mapping
\begin{align}
    (\mu,x)\longmapsto D_\mu u(\mu)(x)
\end{align}
is jointly continuous, we say that $u$ is continuously Lions differentiable and write $u\in C^1(\mcal{P}_2(\bR^d))$.

\subsection{Design of Algorithm}
In this section, we introduce the mean-field control problem and the associated stochastic maximum principle. We then discuss the design of a sample-based numerical algorithm.

Let $U\subset\bR^{d_1}$ denote the control space. We consider the mean-field control problem
\begin{align}
    \inf_{u\in\mcal{U}} J(u), \qquad J(u):=\E\left[\int_0^T f(t,X_t,\mcal{L}(X_t),u_t)dt+g(X_T,\mcal{L}(X_T))\right],
\end{align}
subject to the controlled McKean--Vlasov SDE
\begin{align}
    dX_t=b(t,X_t,\mcal{L}(X_t),u_t)dt+\sigma(t,X_t,\mcal{L}(X_t),u_t)dW_t, \qquad X_0\sim\mu_0,
\end{align}
where $\mu_0\in\mcal{P}_2(\bR^d)$ and $X_0$ is assumed to be independent of $W$. The coefficients are measurable functions
\begin{align}
    (b,\sigma):[0,T]\times\bR^d\times\mcal{P}_2(\bR^d)\times U\rightarrow\bR^d\times\bR^{d\times m}.
\end{align}
The set $\mcal{U}$ of admissible controls consists of all $U$-valued progressively measurable processes $u=(u_t)_{0\leq t\leq T}$ satisfying
\begin{align}
    \E\left[\int_0^T|u_t|^2dt\right]<\infty, \qquad \E\left[\int_0^T\left(|b(t,0,\delta_0,u_t)|^2+|\sigma(t,0,\delta_0,u_t)|^2\right)dt\right]<\infty,
\end{align}
where $\delta_0$ denotes the Dirac measure at $0$.

We impose the following standard assumptions on the coefficients and cost functions, see for instance \cite{Carmona0}, section 4.4. 
\begin{assumption}\label{ass:ass_mfc}
\begin{enumerate}
\item For every $x\in\bR^d$, $\mu\in\mcal{P}_2(\bR^d)$, and $a\in U$,
\begin{align}
    \int_0^T\left(|b(t,x,\mu,a)|^2+|\sigma(t,x,\mu,a)|^2\right)dt<\infty.
\end{align}

\item There exists a constant $C>0$ such that, for every $t\in[0,T]$, $a\in U$, $x_1,x_2\in\bR^d$, and $\mu_1,\mu_2\in\mcal{P}_2(\bR^d)$,
\begin{align}
    |b(t,x_1,\mu_1,a)-b(t,x_2,\mu_2,a)|+|\sigma(t,x_1,\mu_1,a)-\sigma(t,x_2,\mu_2,a)|\leq C\left(|x_1-x_2|+W_2(\mu_1,\mu_2)\right).
\end{align}

\item The running cost $f$ and terminal cost $g$ have at most quadratic growth. More precisely, there exists a constant $C>0$ such that, for every $t\in[0,T]$, $x\in\bR^d$, $\mu\in\mcal{P}_2(\bR^d)$, and $a\in U$,
\begin{align}
    |f(t,x,\mu,a)|\leq C\left(1+|x|^2+|a|^2+\int_{\bR^d}|y|^2\mu(dy)\right), \qquad |g(x,\mu)|\leq C\left(1+|x|^2+\int_{\bR^d}|y|^2\mu(dy)\right).
\end{align}
\end{enumerate}
\end{assumption}
In the following, we review the form of the Stochastic Maximum Principle (SMP) for MFC, and refer readers to \cite{Carmona0} and \cite{Anderson} for more details. We further make the following assumption on top of Assumption \ref{ass:ass_mfc}: 
\begin{assumption}\label{ass:reg}
    For every $t\in[0,T]$, the functions $b(t,\cdot,\cdot,\cdot)$, $\sigma(t,\cdot,\cdot,\cdot)$, and $f(t,\cdot,\cdot,\cdot)$ are continuously differentiable with respect to $(x,u)$ and Lions differentiable with respect to $\mu$, while $g$ is continuously differentiable with respect to $x$ and Lions differentiable with respect to $\mu$. The derivatives $\partial_xb$, $\partial_ub$, $\partial_x\sigma$, $\partial_u\sigma$, $D_\mu b$, and $D_\mu\sigma$ are continuous and uniformly bounded. Moreover, $\partial_xf$, $\partial_uf$, $D_\mu f$, $\partial_xg$, and $D_\mu g$ are continuous and have at most linear growth.
\end{assumption}
We define the Hamiltonian associated with the mean-field control problem by
\begin{align}
    H(t,x,\mu,y,z,u):=f(t,x,\mu,u)+b(t,x,\mu,u)^\top y+\Tr\left(\sigma(t,x,\mu,u)^\top z\right),
\end{align}
for
$(t,x,\mu,y,z,u)\in[0,T]\times\bR^d\times\mcal{P}_2(\bR^d)\times\bR^d\times\bR^{d\times m}\times U$. For a given admissible control $u\in\mcal{U}$, let $X$ denote the corresponding controlled state process and set $\mu_t:=\mcal{L}(X_t)$. The associated adjoint processes $(Y,Z)$ are defined as the solution of the mean-field BSDE
\begin{equation}\label{mf_fbsde}
\begin{cases}
    dY_t=-\left(\partial_xH(t,X_t,\mu_t,Y_t,Z_t,u_t)+\widehat{\E}\left[D_\mu H(t,\widehat{X}_t,\mu_t,\widehat{Y}_t,\widehat{Z}_t,\widehat{u}_t)(X_t)\right]\right)dt+Z_tdW_t,\\
    Y_T=\partial_xg(X_T,\mu_T)+\widehat{\E}\left[D_\mu g(\widehat{X}_T,\mu_T)(X_T)\right],
\end{cases}
\end{equation}
where $(\widehat{X},\widehat{Y},\widehat{Z},\widehat{u})$ denotes an independent copy of $(X,Y,Z,u)$ and $\widehat{\E}$ denotes expectation with respect to the copy variable.

Let $u,v\in\mcal{U}$ and assume that $\mcal{U}$ is convex. Define
\begin{align}
    u_t^\epsilon:=u_t+\epsilon(v_t-u_t),\qquad \beta_t:=v_t-u_t.
\end{align}
Then, by Corollary 4.23 in \cite{Carmona0}, the G\^ateaux derivative of $J$ at $u$ in the direction $\beta$ is given by
\begin{align}
    \left.\frac{d}{d\epsilon}J(u^\epsilon)\right|_{\epsilon=0}
    =\E\left[\int_0^T\partial_uH(t,X_t,\mu_t,Y_t,Z_t,u_t)^\top\beta_t\,dt\right]. \label{eq:gradient_mfc}
\end{align}
Consequently, if $u^*\in\mcal{U}$ is an optimal control, then
\begin{align}
    \E\left[\int_0^T\partial_uH(t,X_t^*,\mu_t^*,Y_t^*,Z_t^*,u_t^*)^\top(v_t-u_t^*)\,dt\right]\geq0,\qquad \forall v\in\mcal{U}. \label{eq:necessary_smp}
\end{align}
This provides the first-order necessary optimality condition and, in particular, gives the gradient information required for the construction of a gradient-based numerical algorithm.

Under an additional convexity condition on the Hamiltonian with respect to the control variable, the variational condition \eqref{eq:necessary_smp} yields the pointwise necessary condition of the stochastic maximum principle:
\begin{align}
    H(t,X_t^*,\mu_t^*,Y_t^*,Z_t^*,u_t^*)=\inf_{a\in U}H(t,X_t^*,\mu_t^*,Y_t^*,Z_t^*,a),\qquad dt\otimes d\mathbb{P}\text{-a.e.} \label{eq:necessary_hamiltonian}
\end{align}
This is the mean-field analogue of the classical Pontryagin stochastic maximum principle; see Theorem 4.24 in \cite{Carmona0}. We next recall the sufficient form of the stochastic maximum principle; see Theorem 4.25 in \cite{Carmona0}.

\begin{theorem}\label{them:sufficient_smp}
Suppose that Assumption \ref{ass:ass_mfc} and \ref{ass:reg} hold. Let $u\in\mcal{U}$ be an admissible control, let $X$ be the corresponding state process, and let $(Y,Z)$ be the associated adjoint processes. Assume further that
\begin{enumerate}
    \item the mapping $(x,\mu)\mapsto g(x,\mu)$ is L-convex;
    \item for $dt\otimes d\mathbb{P}$-a.e. $(t,\omega)$, the mapping $(x,\mu,a)\mapsto H(t,x,\mu,Y_t,Z_t,a)$ is L-convex;
    \item the Hamiltonian is minimized along the controlled trajectory, i.e.,
    \begin{align}
        H(t,X_t,\mu_t,Y_t,Z_t,u_t)=\inf_{a\in U}H(t,X_t,\mu_t,Y_t,Z_t,a),\qquad dt\otimes d\mathbb{P}\text{-a.e.}
    \end{align}
\end{enumerate}
Then $u$ is an optimal control:
\begin{align}
    J(u)=\inf_{v\in\mcal{U}}J(v).
\end{align}
\end{theorem}
\begin{remark}
    A few remarks are in order.
\begin{itemize}
    \item If $b$, $\sigma$, $f$, and $g$ are independent of the law of $X_t$, then all Lions derivative terms vanish and the classical stochastic maximum principle is recovered.
    
    \item The convexity assumptions required by Theorem \ref{them:sufficient_smp} may be restrictive. Nevertheless, the first-order identity \eqref{eq:gradient_mfc} remains useful, as it provides the directional derivative of the objective functional and naturally motivates gradient-based numerical methods.
    
    \item For the purpose of numerical implementation, we restrict attention to controls of the form
\begin{align}
    u_t=\phi(t,X_t,\mu_t),
\end{align}
for some deterministic function $\phi:[0,T]\times\bR^d\times\mcal{P}_2(\bR^d)\rightarrow U$. In the numerical implementation, we consider $\phi$ along the flow of probability measures $\mu_t=\mcal{L}_{X_t}$ induced by the state process and define
\begin{align}
    \bar{\phi}(t,x):=\phi(t,x,\mcal{L}_{X_t}),
\end{align}
so that $u_t=\bar{\phi}(t,X_t)$. The resulting time-state function $\bar{\phi}:[0,T]\times\bR^d\rightarrow U$ is then approximated numerically.
    \end{itemize}

\end{remark}

\subsection{Design of projection algorithm for MFC}
To simplify notation, we write $X_n:=X^N_{t_n}$, $\mu_n:=\mcal{L}_{X_n}$, $u_n:=u_{t_n}$, and $\Delta W_n:=W_{t_{n+1}}-W_{t_n}$. We also use the shorthand
\begin{align}
    b_n:=b(t_n,X_n,\mu_n,u_n),\qquad \sigma_n:=\sigma(t_n,X_n,u_n),\qquad f_n:=f(t_n,X_n,\mu_n,u_n),
\end{align}
and use analogous notation for their derivatives and independent copies. In particular, for an independent copy $(\widehat X_n,\widehat u_n)$ of $(X_n,u_n)$, we write
\begin{align}
    \widehat b_n:=b(t_n,\widehat X_n,\mu_n,\widehat u_n),\qquad \widehat f_n:=f(t_n,\widehat X_n,\mu_n,\widehat u_n).
\end{align}

The first step in the design of the numerical method is to introduce a time-discrete approximation of the mean-field BSDE \eqref{mf_fbsde}. For this part, we assume the diffusion coefficient $\sigma(t,x,u)$ is independent of the distribution, thus we consider the following backward Euler scheme:
\begin{align}
    \begin{cases}
        Y^N_n=\E_{t_n}\left[Y^N_{n+1}+\partial_xH(t_n,X_n,\mu_n,Y^N_{n+1},Z^N_n,u_n)\Delta t\right]+\widehat{\E}\left[D_\mu\widehat b_n(X_n)^\top\widehat Y^N_{n+1}+D_\mu\widehat f_n(X_n)\right]\Delta t,\\
        Z^N_n=\dfrac{1}{\Delta t}\E_{t_n}\left[Y^N_{n+1}(\Delta W_n)^\top\right],\qquad n=0,\ldots,N-1,
    \end{cases}\label{eq:dmfbsde0}
\end{align}
with terminal condition $ Y^N_N=\partial_xg(X_N,\mu_N)+\widehat{\E}\left[D_\mu g(\widehat X_N,\mu_N)(X_N)\right]$. Here $\E_{t_n}[\cdot]:=\E[\cdot\mid\mcal{F}_{t_n}]$, while $\widehat{\E}$ denotes expectation with respect to the independent copy.

Under the Markovian feedback structure introduced above, the conditional expectations appearing in \eqref{eq:dmfbsde0} admit deterministic representations in terms of the current state and its distribution. However, a direct numerical approximation involving the measure argument is difficult to implement. We therefore introduce the following sample-based representation of the discrete adjoint processes:
\begin{align}
    \begin{cases}
        \bar Y_n=\bar Y_{n+1}+\partial_xH(t_n,X_n,\mu_n,\bar Y_{n+1},\bar Z_n,u_n)\Delta t+\widehat{\E}\left[D_\mu\widehat b_n(X_n)^\top\widehat{\bar Y}_{n+1}+D_\mu\widehat f_n(X_n)\right]\Delta t,\\
        \bar Z_n=\dfrac{\bar Y_{n+1}(\Delta W_n)^\top}{\Delta t},
    \end{cases}\label{d_mfbsde1}
\end{align}
again with $ \bar Y_N=\partial_xg(X_N,\mu_N)+\widehat{\E}\left[D_\mu g(\widehat X_N,\mu_N)(X_N)\right]$.

The following result shows that the sample-based variables recover the solution of the discrete mean-field BSDE after conditioning.
\begin{proposition}\label{prop:unbiased_mfbsde}
Let $(Y^N_n,Z^N_n)_{n=0}^{N}$ be defined by \eqref{eq:dmfbsde0} and $(\bar Y_n,\bar Z_n)_{n=0}^{N}$ by \eqref{d_mfbsde1}. Then, for every $n=0,\ldots,N-1$,
\begin{align}
    \E_{t_n}[\bar Y_n]=Y^N_n,\qquad \E_{t_n}[\bar Z_n]=Z^N_n.
\end{align}
Consequently,
\begin{align}
    \E[\bar Y_n]=\E[Y^N_n],\qquad \E[\bar Z_n]=\E[Z^N_n].
\end{align}
\end{proposition}

\begin{proof}
We prove the result by backward induction. Since $\bar Y_N=Y^N_N$, the assertion holds at the terminal time.

Suppose that, for some $n \leq N-1$,
\begin{align}
    \E_{t_{n+1}}[\bar Y_{n+1}]=Y^N_{n+1}.
\end{align}
From the definition of $\bar Z_n$ and the tower property, we obtain
\begin{align}
    \E_{t_n}[\bar Z_n]
    &=\frac{1}{\Delta t}\E_{t_n}\left[\bar Y_{n+1}(\Delta W_n)^\top\right]\\
    &=\frac{1}{\Delta t}\E_{t_n}\left[\E_{t_{n+1}}[\bar Y_{n+1}](\Delta W_n)^\top\right]\\
    &=\frac{1}{\Delta t}\E_{t_n}\left[Y^N_{n+1}(\Delta W_n)^\top\right]
    =Z^N_n.
\end{align}

We next consider the $\bar Y_n$ equation. Since the Hamiltonian is affine in the adjoint variables $(y,z)$, we have
\begin{align}
    \partial_xH(t_n,X_n,\mu_n,y,z,u_n)=\partial_xf_n+\partial_xb_n^\top y+\partial_x\text{tr}(\sigma_n^\top z).
\end{align}
The coefficients $\partial_xb_n$, $\partial_xf_n$, and $\partial_x\sigma_n$ are $\mcal{F}_{t_n}$-measurable. Therefore, using the induction hypothesis and the identity $\E_{t_n}[\bar Z_n]=Z^N_n$,
\begin{align}
    \E_{t_n}\left[\partial_xH(t_n,X_n,\mu_n,\bar Y_{n+1},\bar Z_n,u_n)\right] &= \E_{t_n}\left[ \partial_xf_n+\partial_xb_n^\top \E_{t_n}[\bar Y_{n+1}] +\partial_x\text{tr}(\sigma_n^\top \E_{t_n}[\bar Z_{n}] )\right] \\ 
    &=\E_{t_n}\left[\partial_xH(t_n,X_n,\mu_n,Y^N_{n+1},Z^N_n,u_n)\right].
\end{align}

For the mean-field term, $D_\mu\widehat b_n(X_n)$ and $D_\mu\widehat f_n(X_n)$ are measurable with respect to $\mcal{F}_{t_n}\vee\widehat{\mcal{F}}_{t_n}$. Applying the tower property on the independent-copy probability space gives
\begin{align}
    \widehat{\E}\left[D_\mu\widehat b_n(X_n)^\top\widehat{\bar Y}_{n+1}\right]
    &=\widehat{\E}\left[D_\mu\widehat b_n(X_n)^\top\widehat{\E}_{t_{n+1}}[\widehat{\bar Y}_{n+1}]\right]\\
    &=\widehat{\E}\left[D_\mu\widehat b_n(X_n)^\top\widehat Y^N_{n+1}\right].
\end{align}
Hence, taking $\E_{t_n}$ on both sides of the recursion for $\bar Y_n$ yields
\begin{align}
    \E_{t_n}[\bar Y_n]
    &=\E_{t_n}\left[Y^N_{n+1}+\partial_xH(t_n,X_n,\mu_n,Y^N_{n+1},Z^N_n,u_n)\Delta t\right]\\
    &\quad+\widehat{\E}\left[D_\mu\widehat b_n(X_n)^\top\widehat Y^N_{n+1}+D_\mu\widehat f_n(X_n)\right]\Delta t=Y^N_n.
\end{align}
The result follows by backward induction. Taking expectations once more gives the unconditional identities.
\end{proof}
Meanwhile, the expectation with respect to the independent copy in \eqref{d_mfbsde1} still needs to be approximated numerically. We therefore replace both the law of $X_n$ and the expectation $\widehat{\E}$ by their empirical counterparts. This leads to the following interacting particle system: for each $i\in\{1,\ldots,M\}$,
\begin{align}
\begin{cases}
    X^i_{n+1}=X^i_n+b(t_n,X^i_n,\hat{\mu}_{X_n},u^i_n)\Delta t+\sigma(t_n,X^i_n,u^i_n)\Delta W^i_n,\\
    Y^i_n=Y^i_{n+1}+\partial_xH(t_n,X^i_n,\hat{\mu}_{X_n},Y^i_{n+1},Z^i_n,u^i_n)\Delta t\\
    \qquad\quad+\dfrac{1}{M}\displaystyle\sum_{j=1}^M\left[D_\mu b(t_n,X^j_n,\hat{\mu}_{X_n},u^j_n)(X^i_n)^\top Y^j_{n+1}+D_\mu f(t_n,X^j_n,\hat{\mu}_{X_n},u^j_n)(X^i_n)\right]\Delta t,\\
    Z^i_n=\dfrac{Y^i_{n+1}(\Delta W^i_n)^\top}{\Delta t},
\end{cases}
\label{eq:MFC_sys}
\end{align}
where $\hat{\mu}_{X_n}:=\frac{1}{M}\sum_{l=1}^M\delta_{X^l_n}$. And the terminal condition is approximated accordingly by $Y^i_N=\partial_xg(X^i_N,\hat{\mu}_{X_N})+\frac{1}{M}\sum_{j=1}^M D_\mu g(X^j_N,\hat{\mu}_{X_N})(X^i_N)$, for $i=1,\ldots,M$. 

For feedback controls, we take $u^i_n=\phi(t_n,X^i_n,\hat{\mu}_{X_n})$, or equivalently $u^i_n=\bar{\phi}(t_n,X^i_n)$ along the empirical measure flow.

Once the particle approximations $(X^i_n,Y^i_n,Z^i_n)_{i=1}^M$ have been obtained, the Hamiltonian gradient can be evaluated particlewise as $\partial_uH(t_n,X^i_n,\hat{\mu}_{X_n},Y^i_{n+1},Z^i_n,u^i_n), i=1,\ldots,M.$
These particlewise gradients provide a sample-based approximation of the gradient of the objective functional and are subsequently used to update the control.Based on the preceding discussion, we construct the numerical algorithm in the following three steps.

\begin{enumerate}
\item[\textbf{Step $\mathbf{i}$)}]
In general, a mean-field feedback control takes the form $u(t,x,\mu)$. For the numerical implementation, we restrict the control along the deterministic law flow $\mu_t^u:=\mcal{L}_{X_t^u}$ induced by the fixed initial distribution and the control itself, and define $\bar u(t,x):=u(t,x,\mu_t^u)$.
Abusing notation, we still write $u(t,x)$ for the restricted feedback. Assume that the restricted optimal feedback satisfies $u^*\in\mcal{C}_b^{1,1}([0,T]\times\bR^d;\bR^{d_1})$. Let $0=t_0<t_1<\cdots<t_N=T$ with $\Delta t=T/N$, and approximate the control piecewise constantly in time, i.e., $u_t(\cdot)=u_{t_n}(\cdot)$ for $t\in[t_n,t_{n+1})$. At iteration $k$, the ideal gradient update at $t_n$ is
\begin{align}
    V_{t_n}^{N,k}(X_{t_n}^{N,k}):=u_{t_n}^k(X_{t_n}^{N,k})-\rho_k\partial_uH\left(t_n,X_{t_n}^{N,k},\mcal{L}_{X_{t_n}^{N,k}},Y_{t_{n+1}}^{N,k},Z_{t_n}^{N,k},u_{t_n}^k(X_{t_n}^{N,k})\right). \label{step1}
\end{align}
And for a given feedback control $u^k$, the triple $(X_{t_n}^{N,k},Y_{t_n}^{N,k},Z_{t_n}^{N,k})$ is given by the discrete mean-field forward--backward system
\begin{align}
\begin{cases}
    X^{N,k}_{t_{n+1}}=X^{N,k}_{t_n}+b\left(t_n,X^{N,k}_{t_n},\mcal{L}_{X_{t_n}^{N,k}},u^k_{t_n}(X^{N,k}_{t_n})\right)\Delta t+\sigma\left(t_n,X^{N,k}_{t_n},u^k_{t_n}(X^{N,k}_{t_n})\right)\Delta W_{t_n},\\
    Y^{N,k}_{t_n}=\E_{t_n}\left[Y^{N,k}_{t_{n+1}}+\partial_xH\left(t_n,X^{N,k}_{t_n},\mcal{L}_{X_{t_n}^{N,k}},Y^{N,k}_{t_{n+1}},Z^{N,k}_{t_n},u^k_{t_n}(X^{N,k}_{t_n})\right)\Delta t\right]\\
    \qquad\quad+\widehat{\E}\left[D_\mu H\left(t_n,\widehat X^{N,k}_{t_n},\mcal{L}_{X_{t_n}^{N,k}},\widehat Y^{N,k}_{t_{n+1}},\widehat Z^{N,k}_{t_n},u^k_{t_n}(\widehat X^{N,k}_{t_n})\right)(X^{N,k}_{t_n})\right]\Delta t,\\
    Z^{N,k}_{t_n}=\dfrac{1}{\Delta t}\E_{t_n}\left[Y^{N,k}_{t_{n+1}}\Delta W_{t_n}^\top\right].
\end{cases}
\label{classical_z}
\end{align}

\item[\textbf{Step $\mathbf{ii}$)}]
Since $V_{t_n}^{N,k}$ is only evaluate at sample states, a functional approximation is needed to obtain a feedback control that can be evaluated at arbitrary $x \in \bR^d$. Let $\mcal{M}$ denote an approximation space and define the projection
\begin{align}
    \Pi_\pi^\mcal{M}f\in\arg\min_{\Phi\in\mcal{M}}\|f-\Phi\|_\pi^2,\qquad \|f\|_\pi^2:=\int_{\bR^d}|f(x)|^2\,d\pi(x). \label{projection_operator_general}
\end{align}
Here $\pi_n^k:=\mcal{L}_{X_{t_n}^{N,k}}$. We consider the following two approximation spaces.

\begin{enumerate}
\item For a randomized neural network, let
\begin{align}
    \mcal{M}_L:=\left\{\Phi_r(x)=\sum_{l=1}^Lr_l\phi_l(x):r_l\in\bR^{d_1}\right\},
\end{align}
where $\{\phi_l\}_{l=1}^L$ are fixed randomized basis functions. The corresponding projection is
\begin{align}
    \Pi_{\pi_n^k}^Lf\in\arg\min_{\Phi_r\in\mcal{M}_L}\|f-\Phi_r\|_{\pi_n^k}^2.
\end{align}

\item Alternatively, let $\mcal{M}_{K,(m_1,\ldots,m_K)}^{\rm DNN}$ denote the class of feedforward neural networks
\begin{align}
    \Phi_\theta(x)=W_{K+1}h_K(x)+b_{K+1},\qquad h_j(x)=\rho(W_jh_{j-1}(x)+b_j),\quad j=1,\ldots,K,\qquad h_0(x)=x.
\end{align}
The corresponding nonlinear best approximation is
\begin{align}
    \Pi_{\pi_n^k}^{K,(m_1,\ldots,m_K)}f:=\Phi_{\theta^*},\qquad \theta^*\in\arg\min_\theta\|f-\Phi_\theta\|_{\pi_n^k}^2.
\end{align}
\end{enumerate}

Thus, at the population level, the control update is written generically as
\begin{align}
    u_{t_n}^{k+1}=\Pi_{\pi_n^k}^{\mcal{M}}V_{t_n}^{N,k}. \label{projection_first_stage}
\end{align}

\item[\textbf{Step $\mathbf{iii}$)}]
The distribution $\pi_n^k$ and the conditional expectations in \eqref{classical_z} are generally not available analytically. We therefore introduce an ensemble $\{X_n^{i,k}\}_{i=1}^M$ and the empirical measure $ \widehat\pi_n^k:=\frac{1}{M}\sum_{i=1}^M\delta_{X_n^{i,k}}$.
The population projection in \eqref{projection_first_stage} is then replaced by the empirical least-squares problem 
\begin{align}
    \Pi_{\widehat\pi_n^k}^{M,\mcal{M}}f\in\arg\min_{\Phi\in\mcal{M}}\frac{1}{M}\sum_{i=1}^M|f(X_n^{i,k})-\Phi(X_n^{i,k})|^2.
\end{align}

In our case, the discrete processes $(X_{t_n}^{N,k},Y_{t_n}^{N,k},Z_{t_n}^{N,k})$ are similarly replaced by their particle approximations $(X_n^{i,k},Y_n^{i,k},Z_n^{i,k})$ obtained from \eqref{eq:MFC_sys}. For each particle, set
\begin{align}
    u_n^{i,k}:=u_{t_n}^k(X_n^{i,k}),\qquad \widetilde V_{t_n}^{i,k}:=u_n^{i,k}-\rho_k\partial_uH\left(t_n,X_n^{i,k},\widehat\pi_n^k,Y_{n+1}^{i,k},Z_n^{i,k},u_n^{i,k}\right).
\end{align}
The implementable control update is therefore given by: 
\begin{align}
    u_{t_n}^{k+1}\in\arg\min_{\Phi\in\mcal{M}}\frac{1}{M}\sum_{i=1}^M|\widetilde V_{t_n}^{i,k}-\Phi(X_n^{i,k})|^2. \label{final_control_update}
\end{align}
\end{enumerate}
We summarize the result in Algorithm \ref{algorithm_main2}, and highlight the main idea of the algorithm as follows: 
\begin{itemize}
    \item There is no predefined spatial mesh grids which systematically mitigates the curse of dimensionality issue. In the meantime, solution of the adjoint process (mean field FBSDEs \eqref{mf_fbsde}) is not solved explicitly, and they are only approximated via the samplewise particle approach with the distributions approximated by the relevant particle ensembles at discrete times. 
    \item At each iteration $k+1$, the control at time $t_n$ is obtained based on updating the control function at the $k$-th iteration over the locations $\lbrace X^{k,i}_n \rbrace$ via: $\tilde{u}_n^{k+1, i} := u_n^{k}(X^{k, i}_n) -\eta_k j_n'(u)(X_{n}^{k,i})$ where $j_n'(u)(X_{n}^{k,i})$ is as defined in Step 2 in Algorithm \ref{algorithm_main2}. Since $u^{k+1}_n$ is a function, we then fit a function over those points $\lbrace (X^{k, i}_n,\tilde{u}_n^{k+1, i}) \rbrace^M_{i=1}$ and name it the control $u^{k+1}_n$. Importantly, the sample-wise gradients $j_n'(u)(X_{n}^{k,i})$ are approximately unbiased estimates of the Hamiltonian gradient, so the function approximation task is essentially a regression. 
\end{itemize}
\begin{algorithm}
\footnotesize
\setlength{\abovedisplayskip}{3pt}
\setlength{\belowdisplayskip}{3pt}
\setlength{\abovedisplayshortskip}{2pt}
\setlength{\belowdisplayshortskip}{2pt}
\caption{Main Algorithm (MFC)}\label{algorithm_main2}
\begin{algorithmic}[1]
\REQUIRE Batch size $M$, time steps $N$, terminal time $T$, iterations $K$, learning rates $\{\eta_k\}_{k=0}^{K-1}$, $\Delta t=T/N$; initialize $u_n^0=0$, $n=0,\ldots,N-1$.

\FOR{$k=0,\ldots,K-1$}

\STATE \textbf{Forward simulation:} for $n=0,\ldots,N-1$, define
$\mu_n^k:=\frac{1}{M}\sum_{j=1}^M\delta_{X_n^{k,j}}$ and
$u_n^{k,i}:=u_n^k(X_n^{k,i})$, and simulate, for $i=1,\ldots,M$,
\[
X_{n+1}^{k,i}=X_n^{k,i}+b(t_n,X_n^{k,i},\mu_n^k,u_n^{k,i})\Delta t
+\sigma(t_n,X_n^{k,i},u_n^{k,i})\Delta W_n^{k,i}.
\]

\STATE \textbf{Terminal condition:} with $\mu_N^k:=\frac{1}{M}\sum_{j=1}^M\delta_{X_N^{k,j}}$, set
\[
Y_N^{k,i}=\partial_xg(X_N^{k,i},\mu_N^k)
+\frac{1}{M}\sum_{j=1}^M D_\mu g(X_N^{k,j},\mu_N^k)(X_N^{k,i}),\qquad i=1,\ldots,M.
\]

\STATE \textbf{Backward adjoint simulation:} for $n=N-1,\ldots,0$ and $i=1,\ldots,M$, compute
\begin{align}
    Z_n^{k,i}&=\frac{Y_{n+1}^{k,i}(\Delta W_n^{k,i})^\top}{\Delta t},\\
    Y_n^{k,i}&=Y_{n+1}^{k,i}
    +\partial_xH(t_n,X_n^{k,i},\mu_n^k,Y_{n+1}^{k,i},Z_n^{k,i},u_n^{k,i})\Delta t
+\frac{\Delta t}{M}\sum_{j=1}^M
D_\mu H(t_n,X_n^{k,j},\mu_n^k,Y_{n+1}^{k,j},Z_n^{k,j},u_n^{k,j})(X_n^{k,i}),
\end{align}
and evaluate
\[
j_n'(u)(X_n^{k,i})
:=\partial_uH(t_n,X_n^{k,i},\mu_n^k,Y_{n+1}^{k,i},Z_n^{k,i},u_n^{k,i}).
\]

\STATE \textbf{Control update:} for $n=0,\ldots,N-1$, form the targets
\[
\widetilde u_n^{k+1,i}
=u_n^k(X_n^{k,i})-\eta_k j_n'(u)(X_n^{k,i}),\qquad i=1,\ldots,M,
\]
and fit $u_n^{k+1}\in\mcal{M}$ to
$\{(X_n^{k,i},\widetilde u_n^{k+1,i})\}_{i=1}^M$
using OLS or ridge regression.

\ENDFOR

\RETURN $\{u_n^K\}_{n=0}^{N-1}$.
\end{algorithmic}
\end{algorithm}

\subsection{A special case: Connections between Adjoint Matching and the projection algorithm}
In the mean field control setting for generative modeling, the motivation is to transport via SDE the starting distribution via an SDE such that the distribution of the terminal state of the SDE statistically matches the target distribution. In our formulation, we impose a soft constraint on the distributions of terminal /intermediate state matching.

More specifically,  we formulate the problem as follows: given target measure $ \left(\rho_t \right)_{0 \leq t \leq T} \in \mcal{P}_2(\bR^d)$, minimize the following loss 
\begin{equation}
    J(u; \lambda_f, \lambda_g) = \E \left[\int^T_0 \left( \frac{1}{2}|u_t|^2 + \lambda_f f(\mcal{L}_{X_t}, \rho_t) \right) dt + \lambda_g D(\mcal{L}_{X_T}, \rho_T) \right]  \label{eq:simplifiedLoss}
\end{equation}
subject to the SDE constraint:
\begin{align}
    dX_t=\left( b(t,X_t) + \sigma_t u_t \right)dt + \sigma_t dW_t, \qquad X_0 \sim p_0, 
\end{align}
where $u$ is an admissible control. 
In the formulation above 
\begin{itemize}
    \item $\lambda_g$ is the penalty factor on the terminal state matching. 
    \item $\lambda_f$ is the penalty factor on the intermediate state matching. When it is equal to zero, we obtain a relaxed constraint on target distribution.  Such formulation is the soft-constraint of the Schr{\"o}dinger bridge (SB) formulation \cite{Jiang1}. That is, instead of considering the original SB formulation, the hard terminal constraint is replaced with a soft penalty term $\lambda_g D(\mcal{L}_{X_T}, \rho_T)$ which enters the target loss function.  In this case, it is only required that the generated distribution is statistically similar to $\rho$ at the terminal state. Such formulation could be useful for finetuning purpose since the base $b(t,X_t)$ could be pretrained and then frozen. The new control $u$ needs to be found for minimizing the given loss. 
    \item The terminal loss function $D(\mu, \nu)$, $\mu,\nu \in \mcal{P}_2(\bR^d)$ can take various forms such as the KL divergence or the square of $W_2$ distance. For numerical purposes, different choices can be considered. For instance, 1-D optimal transport problem is easier than its higher dimension counterparts. In higher dimensions, we can take it to be KL divergence and use a differentiable debiased Sinkhorn divergence as a regularized proxy for Wasserstein discrepancy. We will detail the gradient of the terminal loss when we present numerical examples.
\end{itemize}
In in equation \eqref{eq:simplifiedLoss}, if we set $\lambda_f=0$, by using the stochastic maximum principle, we construct the Hamiltonian:
\begin{align}
    H(t, X_t, \mcal{L}_{X_t}, Y_t, Z_t, u_t) = \frac{1}{2}|u_t|^2 + (b(t,X_t)+\sigma_t u_t)^T Y_t + \text{tr}\left(\sigma_t z\right)
\end{align}
where $(Y_t,Z_t)$ are solutions to the corresponding mean field BSDE. We obtain that the critical point takes the form: 
\begin{align}
    u_t= -\sigma^T_t Y_t, \quad \Rightarrow \quad u_{t_n} \approx  -\sigma^T_{t_n} Y^N_{t_n}
\end{align}
where $Y^N_{t_n}$ is as defined in \eqref{eq:dmfbsde0}. And this naturally suggests the following minimizing procedure, at each iteration $k$, freeze the control $u^k$ and minimize over $u^{k+1}$:
\begin{align}
L(u^{k+1};u^k) := \sum^{N-1}_{n=0} \E \left[ |u^{k+1}_{t_n}(t, X^{u^k}_{t_n}) + \sigma^T_t \bar Y^{u^k}_n|^2\right] \label{eq:adjoint_og}
\end{align}
where $(X^{u^k}_{t_n}, \bar Y^{u^k}_n)$ are solutions to \eqref{d_mfbsde1}. We note that this procedure is closely related in spirit to Step 3 of Algorithm \ref{algorithm_main2}. In the setting of \eqref{eq:adjoint_og}, the linear-quadratic structure of the control problem allows the stationarity condition of the Hamiltonian to be characterized explicitly. The control can therefore be recovered recursively through regression, rather than through a separate optimization at each time step.

Importantly, the formulation associated with \eqref{eq:adjoint_og} corresponds to the lean adjoint-matching approach of \cite{DomingoEnrich1}. As one of our theoretical contributions, we will derive in the next section a mean-field analogue of the adjoint-matching formulation in \cite{DomingoEnrich1} and establish its connection with the sample-based mean-field BSDE formulation developed above. In particular, instead of first constructing sample-wise approximations of the time-discretized mean-field BSDE, the lean adjoint-matching approach leads directly to a random pathwise mean-field adjoint ODE. Its time discretization coincides precisely with the sample-wise backward recursion derived from our mean-field BSDE formulation. However, we emphasize that our algorithmic framework is more general in the following respects:
\begin{itemize}
    \item The current lean adjoint-matching formulation is developed for control acting through the drift. In contrast, our construction is based directly on a sample-wise approximation of the mean-field BSDE, in which both $Y^N_{t_n}$ and $Z^N_{t_n}$ are explicitly approximated. Consequently, the framework naturally accommodates control-dependent diffusion coefficients, since the $Z$-component enters the Hamiltonian gradient through the term involving $\partial_u\sigma$.

    \item In the lean adjoint-matching formulation associated with \eqref{eq:adjoint_og}, the loss is constructed by exploiting an explicit form of the critical point of the Hamiltonian. Such a closed-form characterization may not be available for a general control problem. In this case, the gradient information provided by $\partial_uH$ leads naturally to the descent update developed above. This viewpoint is also closely related in spirit to successive-approximation methods, in particular the \emph{Modified Successive Approximation} (MSA) approach.
\end{itemize}
In the next section, recognizing that a complete theoretical analysis of Algorithm \ref{algorithm_main2} is challenging, we focus on the specific class of problems introduced in \eqref{eq:mfpa_Jv_general}. We show that the corresponding continuous-time recursive minimization of the loss \eqref{eq:adjoint_og} yields a self-consistent characterization of the optimal control. This result provides theoretical support for the underlying construction of Algorithm \ref{algorithm_main2}. More importantly, our analysis extends the lean adjoint-matching framework of \cite{DomingoEnrich1} to the mean-field control setting. We further show that the time discretization of the resulting mean-field backward adjoint ODE coincides precisely with the sample-wise backward recursion derived from our mean-field BSDE formulation.

\subsubsection{Adjoint matching: the mean field case}
In this section, we study the MFC problem in the following setup: minimize
\begin{align}
    J(u):=\mathbb E\left[\int_0^T\widetilde f^u(t,X^u_t,\mu^u_t)\,dt+g(X^u_T,\mu^u_T)\right] \label{eq:mfpa_Jv_general}
\end{align}
subject to the SDE
\begin{align}
    dX^u_t &= \widetilde b^u(t,X_t^u,\mu_t^u)\,dt+\sigma_t\,dW_t,\qquad \mu_t^u:=\mcal{L}_{X_t^u}, \label{eq:mfpa_state} \\ 
    \text{where }\widetilde b^u(t,x,\mu) &:= b(t,x,\mu)+\sigma_tu(t,x,\mu), \ \quad  \widetilde f^u(t,x,\mu) := f(t,x,\mu)+\frac{1}{2}|u(t,x,\mu)|^2.
\end{align}
For deep learning purpose, $u(\cdot)$ is approximated by a deep neural network $u^\theta(\cdot)$ where $\theta$ denotes the trainable weights. More specifically, we consider the controlled McKean--Vlasov dynamics
\begin{align}
    dX_t^\theta &= \widetilde b^\theta(t,X_t^\theta,\mu_t^\theta)\,dt+\sigma_t\,dW_t,\qquad \mu_t^\theta:=\mcal{L}_{X_t^\theta}, \label{eq:mfpa_state_theta} \\ 
    \text{where }\widetilde b^\theta(t,x,\mu) &:= b(t,x,\mu)+\sigma_tu^\theta(t,x,\mu). \label{eq:mfpa_tilde_b}
\end{align}
For a fixed smooth vector field $v$, which is regarded as independent of $\theta$ when taking the parameter derivative, define
\begin{align}
    \widetilde f^v(t,x,\mu):=\frac12|v(t,x,\mu)|^2+f(t,x,\mu). \label{eq:mfpa_tilde_f}
\end{align}
and the loss
\begin{align}
    J^v(\theta):=\mathbb E\left[\int_0^T\widetilde f^v(t,X_t^\theta,\mu_t^\theta)\,dt+g(X_T^\theta,\mu_T^\theta)\right]. \label{eq:mfpa_Jv}
\end{align}
In the following proof, we will take
\begin{align}
    v=\operatorname{stopgrad}(u^\theta)
\end{align}
meaning that $v$ will be held fixed when differentiating through $\theta$, and it will eventually be evaluated at $u^\theta$. 
\begin{lemma}[Mean-field pathwise adjoint and gradient formula]
\label{lem:mfpa_pathwise_adjoint}
Assume that $b,f, u^\theta \in C^{1,2}_b([0,T]\times \bR^d \times \mcal{P}_2(\bR^d))$, $u^\theta$ is continuously differentiable in $\theta$ and that $g \in C^{2}_b(\bR^d, \mcal{P}_2(\bR^d))$
Let $v$ be a  fixed smooth vector field , and $X^\theta$ solve \eqref{eq:mfpa_state_theta}. Define the mean-field adjoint
\begin{align}
    a_t=a(t;\mathbb{X}^\theta,u^\theta)
\end{align}
as the unique solution of the backward mean-field ODE
\begin{align}
    \frac{d}{dt}a_t &= -\nabla_x\widetilde b^\theta(t,X_t^\theta,\mu_t^\theta)^Ta_t-\nabla_x\widetilde f^v(t,X_t^\theta,\mu_t^\theta) -\widehat{\mathbb E}\left[D_\mu\widetilde b^\theta(t,\widehat X_t^\theta,\mu_t^\theta)(X_t^\theta)^T\widehat a_t\right]-\widehat{\mathbb E}\left[D_\mu\widetilde f^v(t,\widehat X_t^\theta,\mu_t^\theta)(X_t^\theta)\right], \label{eq:mfpa_adjoint}\\
    a_T &= \nabla_xg(X_T^\theta,\mu_T^\theta) +\widehat{\mathbb E}\left[D_\mu g(\widehat X_T^\theta,\mu_T^\theta)(X_T^\theta)\right]. \label{eq:mfpa_adjoint_terminal}
\end{align}
In the equations above, $(\widehat X^\theta,\widehat a)$ denotes an independent copy of $(X^\theta,a)$ and $\widehat{\mathbb E}$ denotes expectation with respect to the independent copy. The notation $\mathbb{X}^\theta$ denotes that $a_t$ is a path-wise quantity and it is $\mcal{F}_T$-measurable but not adapted to $\mcal{F}_t$.

Then
\begin{align}
    \nabla_\theta J^v(\theta)=\mathbb E\left[\int_0^T\left(\partial_\theta u^\theta(t,X_t^\theta,\mu_t^\theta)\right)^T\sigma_t^Ta_t\,dt\right]. \label{eq:mfpa_Jv_gradient}
\end{align}
Consequently, taking $v=\operatorname{stopgrad}(u^\theta)$ for
\begin{align}
    J(\theta):=\mathbb E\left[\int_0^T\left(\frac12|u^\theta(t,X_t^\theta,\mu_t^\theta)|^2+f(t,X_t^\theta,\mu_t^\theta)\right)\,dt+g(X_T^\theta,\mu_T^\theta)\right], \label{eq:mfpa_Jtheta}
\end{align}
we have
\begin{align}
    \nabla_\theta J(\theta)=\mathbb E\left[\int_0^T\left(\partial_\theta u^\theta(t,X_t^\theta,\mu_t^\theta)\right)^T\left(u^\theta(t,X_t^\theta,\mu_t^\theta)+\sigma_t^Ta_t\right)\,dt\right]. \label{eq:mfpa_gradient_formula}
\end{align}
\end{lemma}
\begin{proof}
Let's first acknowledge that 
\begin{align}
    \left.\nabla_\theta J^v(\theta)\right|_{v=\operatorname{stopgrad}(u^\theta)}=\mathbb E\left[\int_0^T\left(\partial_\theta u^\theta(t,X_t,\mu_t)\right)^T\sigma_t^Ta_t\,dt\right] \label{eq:mfpa_fundamental}
\end{align}
and will prove this identity later. 
Differentiating $J(\theta)$ with respect to $\theta$, we have
   \begin{align}
    \nabla_\theta J(\theta) &= \mathbb E\left[\int_0^T\left(\partial_\theta u^\theta(t,X_t,\mu_t)\right)^Tu^\theta(t,X_t,\mu_t)\,dt\right]+\left.\nabla_\theta J^v(\theta)\right|_{v=\operatorname{stopgrad}(u^\theta)} \\
    &= \mathbb E\left[\int_0^T\left(\partial_\theta u^\theta(t,X_t,\mu_t)\right)^Tu^\theta(t,X_t,\mu_t)\,dt\right]+\mathbb E\left[\int_0^T\left(\partial_\theta u^\theta(t,X_t,\mu_t)\right)^T\sigma_t^Ta_t\,dt\right] \\
    &= \mathbb E\left[\int_0^T\left(\partial_\theta u^\theta(t,X_t,\mu_t)\right)^T\left(u^\theta(t,X_t,\mu_t)+\sigma_t^Ta_t\right)dt\right]. \label{eq:Jtheta_identity}
\end{align} 
For notational simplicity, write
\begin{align}
    X_t:=X_t^\theta,\qquad \mu_t:=\mcal{L}_{X_t^\theta}.
\end{align}
We now show \eqref{eq:mfpa_fundamental}. Differentiating \eqref{eq:mfpa_state_theta} with respect to $\theta$, we have
\begin{align}
    \frac{d}{dt}\partial_\theta X_t^\theta &= \nabla_x\widetilde b^\theta(t,X_t,\mu_t)\partial_\theta X_t^\theta+\partial_\theta\widetilde b^\theta(t,X_t,\mu_t) +\widehat{\mathbb E}\left[D_\mu\widetilde b^\theta(t,X_t,\mu_t)(\widehat X_t)\widehat{\partial_\theta X_t^\theta}\right],\qquad \partial_\theta X_0^\theta=0. \label{eq:mfpa_variational}
\end{align}
Note that by definition of $\widetilde b^\theta$, we have $ \partial_\theta\widetilde b^\theta(t,x,\mu)=\sigma_t\partial_\theta u^\theta(t,x,\mu)$. 
Now, differentiating $J^v(\theta)$ with respect to $\theta$ gives:
\begin{align}
    \nabla_\theta J^v(\theta) &= \mathbb E\left[\int_0^T(\partial_\theta X_t^\theta)^T\nabla_x\widetilde f^v(t,X_t,\mu_t)\,dt+(\partial_\theta X_T^\theta)^T\nabla_xg(X_T,\mu_T)\right] \\
    &\quad+\mathbb E\left[\int_0^T\widehat{\mathbb E}\left[\widehat{\partial_\theta X_t^\theta}^TD_\mu\widetilde f^v(t,X_t,\mu_t)(\widehat X_t)\right]dt\right]+\mathbb E\left[\widehat{\mathbb E}\left[\widehat{\partial_\theta X_T^\theta}^TD_\mu g(X_T,\mu_T)(\widehat X_T)\right]\right]. \label{eq:mfpa_Jv_raw}\\
    &= \mathbb E\left[\int_0^T(\partial_\theta X_t^\theta)^T\left(\nabla_x\widetilde f^v(t,X_t,\mu_t) +\widehat{\mathbb E}\left[D_\mu\widetilde f^v(t,\widehat X_t,\mu_t)(X_t)\right]\right)dt\right] \\
    &\quad+\mathbb E\left[(\partial_\theta X_T^\theta)^T\left(\nabla_xg(X_T,\mu_T) +\widehat{\mathbb E}\left[D_\mu g(\widehat X_T,\mu_T)(X_T)\right]\right)\right] \label{eq:mfpa_Jv_exchange}\\
    &=\mathbb E\left[\int_0^T(\partial_\theta X_t^\theta)^T\left(\nabla_x\widetilde f^v(t,X_t,\mu_t) +\widehat{\mathbb E}\left[D_\mu\widetilde f^v(t,\widehat X_t,\mu_t)(X_t)\right]\right)dt + (\partial_\theta X_T^\theta)^T a_T\right] \label{eq:mfpa_Jv_terminal}
\end{align}
In the sequence of equations above, \eqref{eq:mfpa_Jv_exchange} is due to Fubini's theorem and \eqref{eq:mfpa_Jv_terminal} is due to the definition of $a_T$. 
Now differentiate $a_t \partial_\theta X_t^\theta$ under the expectation, using integration by part, we get: 
\begin{align}
    \frac{d}{dt}\mathbb E[a_t^T\partial_\theta X_t^\theta] &= \mathbb E\left[(\partial_\theta X_t^\theta)^T\frac{d}{dt}a_t+(\partial_\theta X_t^\theta)^T\nabla_x\widetilde b^\theta(t,X_t,\mu_t)^Ta_t+a_t^T\partial_\theta\widetilde b^\theta(t,X_t,\mu_t)\right] \\
    &\quad+\mathbb E\widehat{\mathbb E}\left[a_t^TD_\mu\widetilde b^\theta(t,X_t,\mu_t)(\widehat X_t)\widehat{\partial_\theta X_t^\theta}\right]. \label{eq:mfpa_product_raw} \\
    &= \mathbb E\left[(\partial_\theta X_t^\theta)^T\left(\frac{d}{dt}a_t+\nabla_x\widetilde b^\theta(t,X_t,\mu_t)^Ta_t+\widehat{\mathbb E}\left[D_\mu\widetilde b^\theta(t,\widehat X_t,\mu_t)(X_t)^T\widehat a_t\right]\right)\right] \\
    &\quad+\mathbb E\left[a_t^T\partial_\theta\widetilde b^\theta(t,X_t,\mu_t)\right]. \label{eq:mfpa_product_exchange}
\end{align}
where in \eqref{eq:mfpa_product_exchange} we again used the Fubini's theorem. 
Hence, using \eqref{eq:mfpa_adjoint}, we have 
\begin{align}
    \frac{d}{dt}\mathbb E[a_t^T\partial_\theta X_t^\theta] &= -\mathbb E\left[(\partial_\theta X_t^\theta)^T\left(\nabla_x\widetilde f^v(t,X_t,\mu_t)+\widehat{\mathbb E}\left[D_\mu\widetilde f^v(t,\widehat X_t,\mu_t)(X_t)\right]\right)\right]+\mathbb E\left[a_t^T\partial_\theta\widetilde b^\theta(t,X_t,\mu_t)\right]. \label{eq:mfpa_cancellation}
\end{align}
Integrating over $[0,T]$ and using $\partial_\theta X_0^\theta=0$ gives
\begin{align}
    \mathbb E[(\partial_\theta X_T^\theta)^Ta_T] =-\mathbb E\left[\int_0^T(\partial_\theta X_t^\theta)^T\left(\nabla_x\widetilde f^v(t,X_t,\mu_t)+\widehat{\mathbb E}\left[D_\mu\widetilde f^v(t,\widehat X_t,\mu_t)(X_t)\right]\right)dt\right] + \mathbb E\left[\int_0^T\left(\partial_\theta\widetilde b^\theta(t,X_t,\mu_t)\right)^Ta_t\,dt\right]. \label{eq:mfpa_main_identity}
\end{align}
Substituting this equation in \eqref{eq:mfpa_Jv_terminal}, we have
\begin{align}
    \nabla_\theta J^v(\theta)&=\mathbb E\left[\int_0^T\left( \partial_\theta\widetilde b^\theta(t,X_t,\mu_t)\right)^Ta_t\,dt\right] \\
    &=\mathbb E\left[\int_0^T\left(\partial_\theta u^\theta(t,X_t,\mu_t)\right)^T\sigma_t^Ta_t\,dt\right]
\end{align}
which proves \eqref{eq:mfpa_fundamental}. 
\end{proof}

Next, at $v=\operatorname{stopgrad}(u^\theta)$, based on the identity \eqref{eq:Jtheta_identity},
\begin{align}
    \nabla_\theta J(\theta)&=\mathbb E\left[\int_0^T\left(\partial_\theta u^\theta(t,X_t^v,\mu_t^v)\right)^T\left(u^\theta(t,X_t^v,\mu_t^v)+\sigma_t^T a(t;\mathbb{X}^{v},v)\right)\,dt\right] \\
    &=\mathbb E\left[ \frac{1}{2} \int^T_0 \nabla_\theta |u^\theta(t,X^{v},\mcal{L}_{X^{{v}}})|^2 dt + \int^T_0  \left( \nabla_\theta u^\theta(t,X^{v},\mcal{L}_{X^{v}}) \right)^T \sigma^T_t a(t;\mathbb{X}^{v},v) dt \right] \\
    & = \frac{1}{2}  \E \left[\nabla_\theta \int^T_0 |u^\theta(t,X^{v},\mcal{L}_{X^{v}}) + \sigma^T_t a(t;\mathbb{X}^{v}, v) |^2 dt \right]. \label{eq:a_path}
\end{align}
\subsubsection*{The Basic Adjoint Matching}
Motivated by \eqref{eq:a_path}, we define the \emph{basic mean-field adjoint matching loss}
\begin{align}
    \mcal{L}_{\rm B-MFAM}(v;u):=\frac12\int_0^T\left|v(t,X_t^u,\mu_t^u)+\sigma_t^Ta(t;\mathbb{X}^u,u)\right|^2dt. \label{eq:mfam_loss}
\end{align}
When $u=u^\theta$ is parametrized by $\theta$, the stop-gradient convention $u=\mathsf{stopgrad} u^\theta$ yields
\begin{align}
    \nabla_\theta^{\rm sg}\E\left[\mcal{L}_{\rm B-MFAM}(u^\theta;u)\right]=\nabla_\theta J(\theta). \label{eq:mfam_gradient_equivalence}
\end{align}
As such, the gradient of the stochastic optimal control objective can be recovered from a regression-type loss in which the reference feedback, forward trajectory, and adjoint are frozen. In what follows, we show that every functional critical point satisfying \eqref{eq:functional_critical} induces the unique optimal control process. 
\begin{align}
    u(t,X^u_t,\mu^u_t) = -\sigma^T_t \E[a(t;\mathbb{X}^{ u}, u)|X^{u}_t] \label{eq:functional_critical}
\end{align}
 We emphasize that, from this point onward, we no longer restrict the control to a parametrization $u^\theta$. Consequently, the critical-point condition derived below is understood in the functional sense, rather than with respect to a particular class such as neural networks. For a generic control $u$ which is sufficiently regular, we abuse notation by still using $a_t$ for the the mean-field adjoint $a_t=a(t;\mathbb{X}^u,u)$ as the unique solution of the backward mean-field ODE: 
\begin{align}
    \frac{d}{ds}a_s &= -\nabla_x\widetilde b^u(s,X_s,\mu_s)^Ta_s-\nabla_x\widetilde f^u(s,X_s,\mu_s) -\widehat{\mathbb E}\left[D_\mu\widetilde b^u(s,\widehat X_s,\mu_s)(X_s)^T\widehat a_s\right]-\widehat{\mathbb E}\left[D_\mu\widetilde f^u(s,\widehat X_s,\mu_s)(X_s)\right], \quad \forall s \in [06,T]\\
    a_T &= \nabla_xg(X_T,\mu_T) +\widehat{\mathbb E}\left[D_\mu g(\widehat X_T,\mu_T)(X_T)\right] \label{eq:a_continuous}. 
\end{align}

For $dt$-a.e. $t$, the $L^2$ minimization of $\E[\mcal{L}_{\rm B-MFAM}(v;u)]$ is the projection problem. Hence, for $\mu_t^u$-almost every $x$,
\begin{align}
    v(t,x,\mu_t^u)&=-\sigma_t^T\E\left[a(t;\mathbb{X}^u,u)\mid X_t^u=x\right]
    \label{eq:mfam_regression_target}.
\end{align}
To obtain the adjoint-matching fixed point, we set the reference control equal to the current control while stopping the gradient through the reference argument. More precisely, let $\bar u:=\mathsf{stopgrad}(u)$ and consider
\begin{align}
    \E[\mcal{L}_{\rm B- MFAM}(u;\bar u)].
\end{align}
At a functional stop-gradient critical point, $\bar u$ and $u$ have the same numerical value, while the dependence through $X^{\bar u}$ and $a(t;\mathbb{X}^{\bar u},\bar u)$ is frozen under differentiation. Hence \eqref{eq:mfam_regression_target} becomes the self-consistency relation
\begin{align}
    u(t,x,\mu_t^u)
    &=-\sigma_t^T\E\left[a(t;\mathbb{X}^u,u)\mid X_t^u=x\right] \qquad \mu_t^u\text{-a.e. }x.
    \label{eq:mfam_self_consistency_flow}
\end{align}

Next,  we establish that such control process \eqref{eq:mfam_self_consistency_flow} coincides with the optimal control of the original problem via the SMP.  First, denote $\E_{t}[\cdot]:= \E[\cdot|\mcal{F}_{t}]$ and that $\bar Y_t:=\E_{t}[a_t]$. Since the controlled McKean--Vlasov dynamics are Markovian in $(X_t^u,\mu_t^u)$ and, in the absence of common noise, the law flow $\mu_t^u=\mcal L_{X_t^u}$ is deterministic, the conditional projection admits a Markovian representation. Hence,
\begin{align}
    \bar Y_t:=\E[a_t\mid\mcal F_t] =\E[a_t\mid X_t^u]. \label{eq:conditional_adjoint_markov}
\end{align}
We have that by \eqref{eq:a_continuous} and the tower property:
\begin{align}
    \E_{t}[a_t]
    &= \E_t\Big[a_T + \int^T_t \nabla_x\widetilde b^u(s,X_s,\mu_s)^T \E_s[a_s]+\nabla_x\widetilde f^u(s,X_s,\mu_s)+\\
    &+\widehat{\mathbb E}\big[D_\mu\widetilde b^u(s,\widehat X_s,\mu_s)(X_s)^T \widehat{\E}_s[\widehat a_s] \big]+ \widehat{\mathbb E}\big[D_\mu\widetilde f^u(s,\widehat X_s,\mu_s)(X_s)\big] ds\Big] \label{eq:mart_1}
\end{align}
Denote $\nabla_x\widetilde b^u(t,X_t,\mu_t)^T \bar Y_t+\nabla_x\widetilde f^u(t,X_t,\mu_t)+ \widehat{\mathbb E}\left[D_\mu\widetilde b^u(t,\widehat X_t,\mu_t)(X_t)^T \widehat{\bar Y_t} \right]+ \widehat{\mathbb E}\left[D_\mu\widetilde f^u(t,\widehat X_t,\mu_t)(X_t)\right]:=F_t$, we have: 
\begin{align}
    \eqref{eq:mart_1}= \E_t[a_T + \int^T_0 F_s ds - \int^t_0 F_s ds] \\
    \Rightarrow \bar Y_t = M_t - \int^t_0 F_s ds
\end{align}
where $M_t:=\E_t[a_T + \int^T_0 F_s ds]$ is a square integrable martingale. 

Then, by the Martingale representation theorem and definition of $\bar Y$, there exists a predictable square-integrable process $\bar Z_t$ such that: 
\begin{align}
    d\bar Y_t =-\left(\nabla_x\widetilde b^u(t,X_t,\mu_t)^T \bar Y_t+\nabla_x\widetilde f^u(t,X_t,\mu_t)+ \widehat{\mathbb E}\left[D_\mu\widetilde b^u(t,\widehat X_t,\mu_t)(X_t)^T \widehat{\bar Y_t} \right]+ \widehat{\mathbb E}\left[D_\mu\widetilde f^u(t,\widehat X_t,\mu_t)(X_t)\right] \right)dt + \bar{Z}_t dW_t \label{eq:bar_y_1st}
\end{align}
with the terminal condition $\bar Y_T = a_T$.
Recall that
\begin{align}
    \nabla_x\widetilde b^u =\nabla_xb+\sigma_t\nabla_xu, \qquad \nabla_x\widetilde f^u =\nabla_xf+(\nabla_xu)^Tu.
\end{align}
The criticality condition gives $(\nabla_x u)^T (\sigma^T_t \bar Y_t +u)=0$.
Similarly, we have
\begin{align}
    \widehat\E\left[D_\mu u(t,\widehat X_t^u,\mu_t^u)(X_t^u)^T\left(\widehat u_t+\sigma_t^T\widehat{\overline Y}_t\right)\right]=0.
\end{align}
As such, at the critical point we have the following adjoint equation: 
\begin{equation}
    d\overline Y_t
    =-\big(\nabla_xb(t,X_t^u,\mu_t^u)^T\overline Y_t+\nabla_xf(t,X_t^u,\mu_t^u)+\widehat\E\big[
D_\mu b(t,\widehat X_t^u,\mu_t^u)(X_t^u)^T\widehat{\overline Y}_t+D_\mu f(t,\widehat X_t^u,\mu_t^u)(X_t^u)\big] \big)dt+\overline Z_tdW_t
\end{equation}
\noindent with terminal condition $\overline Y_T=\nabla_xg(X_T^u,\mu_T^u)+\widehat\E[D_\mu g(\widehat X_T^u,\mu_T^u)(X_T^u)]$. We note that it is exactly the adjoint mean-field BSDE in the setup of the Stochastic Maximum principle. Meanwhile, let the assumptions made in \cite{carmona2018probabilistic1} section 6.4.1 hold. Then by Theorem~6.19 therein, with the choice $u_t=-\sigma^T_t \bar Y_t$,  $(\overline X_t, \overline Y_t ,\overline Z_t)$  is the unique square-integrable adapted solution to the following coupled mean-field BSDE given the same terminal condition and $X_0 \sim \mu_0$. 
\begin{align}
    \begin{cases}
        dX_t = \left( b(t,X_t, \mu_t) - \sigma_t \sigma^T_t \overline Y_t \right) dt + \sigma_t dW_t \\ 
        d\overline Y_t  =-\Big(\nabla_xb(t,X_t,\mu_t)^T\overline Y_t+\nabla_xf(t,X_t,\mu_t)+\widehat\E\left[
D_\mu b(t,\widehat X_t,\mu_t)(X_t)^T\widehat{\overline Y}_t+D_\mu f(t,\widehat X_t,\mu_t)(X_t)\right] \Big)dt+\overline Z_tdW_t
    \end{cases}
\end{align}

The next theorem immediately shows that the self-consistent critical point of the expected basic adjoint matching loss is the unique optimal control process to the original problem. 
\begin{theorem}\label{thm:mfam_optimality}
Assume that the regularity conditions required above hold and that the convexity assumptions of the sufficient mean-field stochastic maximum principle are satisfied\footnote{The assumption made earlier as in \cite{carmona2018probabilistic1} again suffices. Assume $f$ satisfies the $L-$convexity assumption. The $L-$convexity condition is then satisfied naturally for $\tilde f$ because in our case: $\tilde f(t,x,\mu,a):=f(t,x,\mu)+\frac{1}{2}|a|^2$, and so $\tilde f(t,x',\mu',a')-\tilde f(t,x,\mu,a) -\partial_x \tilde f(t,x,\mu,a)(x'-x) - \widehat{\E}[D_\mu \tilde f(t,x,\mu,a)(\widehat{X}) (\widehat{X'}-\widehat{X})] - \partial_a \tilde f(t,x,\mu,a)(a'-a) \geq \frac{1}{2}|a'-a|^2 $.} in particular, $(x,\mu)\mapsto g(x,\mu)$ is L-convex and, for $dt\otimes d\mathbb{P}$-a.e. $(t,\omega)$,
\begin{align}
    (x,\mu,a)\mapsto H(t,x,\mu,Y_t^u,Z_t^u,a)
\end{align}
is jointly L-convex. Then every functional stop-gradient critical point of the expected basic adjoint matching loss induces the unique optimal control process.
\end{theorem}
\begin{proof}
    Under our current setup, the Hamiltonian takes the form: 
\begin{align}
    H(t,x,\mu,y,z,a):=f(t,x,\mu)+\frac{1}{2}|a|^2+(b(t,x,\mu)+\sigma_ta)^Ty+\rm Tr(\sigma_t^Tz).
\end{align}
It is strictly convex in the control variable and when take $u_t = -\sigma^T_t Y_t$, the minimization of  
\begin{align}
    \text{argmin}_{a \in \bR^{d_1}} H(t,X^u_t,\mu^u_t,Y^u_t,Z^u_t,a)
\end{align}
is attained by at $a=-\sigma^T_t Y^u_t$. Hence, by \ref{them:sufficient_smp}, we conclude that the self-consistent critical point of the expected basic mean-field adjoint matching loss is an optimal control of the MFC problem. To obtain uniqueness, note that under our setup, 
\begin{align}
    H(t,x',\mu',y,z,a')- H(t,x,\mu,y,z,a) &-  \partial_a H(t,x,\mu,y,z,a)^T(a'-a) - \partial_x H(t,x,\mu,y,z,a)^T(x'-x)\\
    -\widehat \E[D_\mu H(t,x,\mu,y,z,a)(\widehat{X}) (\widehat{X}' -\widehat{X})] \geq \frac{1}{2}|a'-a|^2
\end{align}
Then in the convention of \cite{carmona2018probabilistic1}, the Hamiltonian is strongly convex in $a$ with $\lambda=\frac{1}{2}$. Then uniqueness follows by Remark 6.17 \cite{carmona2018probabilistic1} and the proof that follows. 
\end{proof}
\subsubsection*{The Lean Adjoint Matching}
Recall that, for a given feedback control $u$,
\begin{align}
    \widetilde b^u(t,x,\mu)
    &:=b(t,x,\mu)+\sigma_tu(t,x,\mu),\\
    \widetilde f^u(t,x,\mu)
    &:=f(t,x,\mu)+\frac{1}{2}|u(t,x,\mu)|^2.
\end{align}
The basic mean-field adjoint equation therefore involves the derivatives $\nabla_x\widetilde b^u,D_\mu\widetilde b^u,\nabla_x\widetilde f^u,D_\mu\widetilde f^u$.
Although the derivatives of the prescribed functions $b$ and $f$ can be
computed independently of the optimization iterations, the frozen feedback $\bar u$ changes after each control update. Consequently, evaluating the basic mean-field adjoint requires repeatedly computing the state and Lions derivatives $\nabla_x\bar u$ and $D_\mu\bar u$. To avoid these additional control-derivative calculations, we introduce a \emph{lean mean-field adjoint matching} procedure, in which the terms involving $\nabla_x\bar u$ and $D_\mu\bar u$ are removed from the backward adjoint equation while the derivatives of the base drift $b$ and running cost $f$ are retained:
\begin{align}
     \mcal{L}_{\rm L-MFAM}(w;u)& :=\frac{1}{2}\int_0^T\left|w(t,X_t^{ u},\mu_t^{u})+\sigma_t^T \tilde{a}(t;\mathbb{X}^{u}, u)\right|^2dt, \quad \mathbb{X}^u \sim p^{u}  \quad u :=\mathsf{stopgrad}(w)\\ 
     \frac{d}{dt} \tilde a_t &= -\nabla_x b(t,X_t,\mu_t)^T \tilde a_t-\nabla_x f(t,X_t,\mu_t) -\widehat{\mathbb E}\left[D_\mu  b(t,\widehat X_t,\mu_t)(X_t)^T\widehat{\tilde{a}_t}\right]-\widehat{\mathbb E}\left[D_\mu  f(t,\widehat X_t,\mu_t)(X_t)\right]\\
    \tilde a_T &= \nabla_xg(X_T,\mu_T) +\widehat{\mathbb E}\left[D_\mu g(\widehat X_T,\mu_T)(X_T)\right] \label{eq:lean_mf_adjoint}.
\end{align}
In the following paragraphs, we show that every functional critical point of the expected lean mean-field adjoint matching loss is also a functional critical point of the basic mean-field adjoint matching loss. It then follows from Theorem \ref{thm:mfam_optimality} that such a critical point coincides with the optimal control of the original mean-field control problem.  For that, we first show a lemma. 
\begin{lemma}\label{thm:lemma_uniqueness}
    Let $u$ be an arbitrary control. Assume that $\tilde b(t,x,\mu)=b(t,x,\mu)+\sigma_t u(t,x,\mu)$ is continuously differentiable in $x$ and Lions differentiable in $\mu$. And that 
    \begin{align}
        |\nabla \tilde b(t,x,\mu)| & \leq L_x, \qquad |D_\mu \tilde b(t,x,\mu)(y)| \leq L_\mu.
    \end{align}
Then the following has at most one square-integrable solution. 
\begin{align}
    Y_t &= \E \Bigg[ \int^T_t \Big( \nabla \tilde b(s,X^{u}_s,\mu^{u}_s)^T Y_s+\nabla \tilde f(s,X^{u}_s,\mu^{u}_s) +\widehat{\mathbb E}\left[D_\mu \tilde{b}(s,\widehat{X^{u}_s},\widehat{\mu^{u}_s})(X^{u}_s)^T \hat{Y}_s\right]\\
    &+ \widehat{\mathbb E}\left[D_\mu \tilde f(s,\widehat{X^{u}_s},\widehat{\mu^{u}_s})(X^{u}_s)\right] \Big) ds +\nabla g(X^{u}_T, \mu^{u}_T)+ \widehat{\mathbb E} [D_\mu g(\widehat{X^{u}_T}, \widehat{\mu^{u}_T})(X^{u}_T)] \bigg| X^{u}_t \Bigg] \label{eq:adjoint_equiv_lemma}.
\end{align}
\end{lemma}
\begin{proof}
    Let $Y^1, Y^2$ be two solutions of \eqref{eq:adjoint_equiv_lemma}, then we take the difference between the two to get: 
    \begin{align}
        Y_t^1-Y_t^2=  \E \Bigg[ \int^T_t \nabla \tilde b(s,X^{u}_s,\mu^{u}_s)^T (Y^1_s-Y^2_s)+\widehat{\mathbb E}\left[D_\mu \tilde{b}(s,\widehat{X^{u}_s},\widehat{\mu^{u}_s})(X^{u}_s)^T (\hat{Y}^1_s-\hat{Y}^2_s)\right] ds \bigg| X^{u}_t \Bigg]
    \end{align}
This gives that by Cauchy's inequality: 
\begin{align}
    |Y_t^1-Y_t^2| &\leq \int^T_t \E \Big[ |\nabla \tilde b(s,X^{u}_s,\mu^{u}_s)| |Y^1_s-Y^2_s| ds \big| X^{u}_t \Big]ds +
    \int^T_t \E \left[ \widehat{\mathbb E}\left[|D_\mu \tilde{b}(s,\widehat{X^{u}_s},\widehat{\mu^{u}_s})(X^{u}_s)| |\hat{Y}^1_s-\hat{Y}^2_s| \right]  ds \big| X^{u}_t \right]ds \\
    &\leq \int^T_t \E \Big[ |\nabla \tilde b(s,X^{u}_s,\mu^{u}_s)|^2 \big| X^{u}_t \Big]^{\frac{1}{2}} \E \Big[  |Y^1_s-Y^2_s|^2  \big| X^{u}_t \Big]^{\frac{1}{2}} ds \\
    &+ \int^T_t \E \left[ \widehat{\mathbb E}\left[|D_\mu \tilde{b}(s,\widehat{X^{u}_s},\widehat{\mu^{u}_s})(X^{u}_s)|^2 \right]   \big| X^{u}_t \right]^{\frac{1}{2}} \E \left[|Y^1_s-Y^2_s|^2 \right]^{\frac{1}{2}}  ds
\end{align}
In the last step, we used the fact that $\E \left[|Y^1_s-Y^2_s|^2 \right] = \widehat{\E} \left[|\hat Y^1_s-\hat Y^2_s|^2 \right]$.
Then, squaring both sides and use Cauchy's inequality and boundness assumption on $\nabla b, D_\mu b$ we have
\begin{align}
   |Y_t^1-Y_t^2|^2 & \leq 2 \int^T_t \E \Big[ |\nabla \tilde b(s,X^{u}_s,\mu^{u}_s)|^2 \big| X^{u}_t \Big] ds \int^T_t \E \Big[  |Y^1_s-Y^2_s|^2  \big| X^{u}_t \Big] ds \\
    &+ 2\int^T_t \E \left[ \widehat{\mathbb E}\left[|D_\mu \tilde{b}(s,\widehat{X^{u}_s},\widehat{\mu^{u}_s})(X^{u}_s)|^2 \right]  \big| X^{u}_t \right] ds \int^T_t \E \left[|Y^1_s-Y^2_s|^2 \right]  ds \\ 
    & \leq 2(T-t) L^2_{x} \int^T_t \E \Big[  |Y^1_s-Y^2_s|^2  \big| X^{u}_t \Big] ds +2(T-t) L^2_\mu \int^T_t  \E \Big[  |Y^1_s-Y^2_s|^2 \Big] ds. 
\end{align}
Taking expectation on both sides, we obtain: 
\begin{align}
    \E[|Y_t^1-Y_t^2|^2] \leq 2T(L^2_x+L^2_\mu)\int^T_t \E \Big[  |Y^1_s-Y^2_s|^2 \Big] ds.
\end{align}
Then by the backward Gronwall's inequality, we have: 
\begin{align}
   \E[|Y^1_t - Y^2_t|^2]=0 \qquad \Rightarrow \qquad Y^1_t = Y^2_t, \quad \text{a.s. \ } \forall t \in [0,T]. 
\end{align}
\end{proof}
\begin{proposition}
Let $u'$ be a functional stop-gradient critical point of
$\E[\mcal{L}_{\rm L-MFAM}(w;u)]|_{u=w}$ and assume that
$\widetilde b^{u'}$ satisfies the regularity and boundedness assumptions of
Lemma~\ref{thm:lemma_uniqueness}. Then $u'$ is also a functional
stop-gradient critical point of the basic mean-field adjoint matching loss.
Consequently, by Theorem~\ref{thm:mfam_optimality}, $u'$ coincides with the
optimal feedback control $u^*$ of the original mean-field control problem.
\end{proposition}
\begin{proof}
    We note that the lean adjoint loss can be written as follows: 
    \begin{align}
        \E[\mcal{L}_{\rm L-MFAM}(w;u)] & = \frac{1}{2}\E\left[\int_0^T\left|w(t,X_t^{ u},\mu_t^{u})+\sigma_t^T \tilde{a}(t;\mathbb{X}^{u}, u)\right|^2dt\right] \\
        & = \frac{1}{2}\E\left[\int_0^T\left|w(t,X_t^{u},\mu_t^{u})+\sigma_t^T \E[\tilde{a}(t;\mathbb{X}^{u},u) \big | X^u_t ] \right|^2dt + \int_0^T \left|\sigma_t^T \left( \tilde{a}(t;\mathbb{X}^{u},u)- \E[\tilde{a}(t;\mathbb{X}^{u},u) \big | X^u_t ] \right) \right|^2dt\right].
    \end{align}
Taking the first variation with respect to $w$ and set it to zero, we obtain the following condition for the critical point $u'$: 
\begin{align}
    u'(t,X_t^{u'},\mu_t^{u'})= - \sigma_t^T \E[\tilde{a}(t;\mathbb{X}^{u'},u') \big | X^{u'}_t ]
\end{align}
Next, note that 
\begin{align}
   \frac{1}{2}  D_\mu \E[\left|{u}'(t,X^{u'}_t, \mu^{u'}_t) \right|^2](X^{u'}_t) &=  \nabla u'(t,X^{u'}_t, \mu^{u'}_t)^T u' + \widehat{\E}[D_\mu u'(t, \widehat{X^{u'}_t}, \widehat{\mu^{u'}_t})^T(X^{u'}_t) u'(t, \widehat{X^{u'}_t}, \widehat{\mu^{u'}_t}) ] \\
   & = - \nabla u'(t,X^{u'}_t, \mu^{u'}_t)^T \sigma_t^T \E[\tilde{a}(t;\mathbb{X}^{u'},u') \big | X^{u'}_t] - \widehat{\E}\left[D_\mu u'(t, \widehat{X^{u'}_t}, \widehat{\mu^{u'}_t})^T(X^{u'}_t) \sigma_t^T \widehat{\E}[\widehat{\tilde{a}_t} \big | \widehat{X^{u'}_t} ] \right] .
\end{align}
And so we have for $\forall t \in [0,T]$: 
\begin{align}
    \nabla u'(t,X^{u'}_t, \mu^{u'}_t)^T u' &+ \widehat{\E}[D_\mu u'(t, \widehat{X^{u'}_t}, \widehat{\mu^{u'}_t})^T(X^{u'}_t) u'(t, \widehat{X^{u'}_t}, \widehat{\mu^{u'}_t}) ] + \nabla u'(t,X^{u'}_t, \mu^{u'}_t)^T \sigma_t^T \E[\tilde{a}(t;\mathbb{X}^{u'},u') \big | X^{u'}_t] \\
    &+\widehat{\E}\left[D_\mu u'(t, \widehat{X^{u'}_t}, \widehat{\mu^{u'}_t})^T(X^{u'}_t) \sigma_t^T \widehat{\E}[\widehat{\tilde{a}_t} \big | \widehat{X^{u'}_t} ] \right]=0 \label{eq:u_diff}.
\end{align}
Then at $u=u'$ the critical point, writing $\tilde{a}'_t:=\tilde{a}(t;\mathbb{X}^{u'},u')$ we have: 
\begin{align}
    \E[\tilde{a}_t'\big | X^{u'}_t]&=\E\Bigg[ \int^T_t \nabla b(s,X^{u'}_s,\mu^{u'}_s)^T \tilde{a}_s +\nabla f(s,X^{u'}_s,\mu^{u'}_s)+\widehat{\mathbb E}\left[D_\mu b(s,\widehat{X^{u'}_s},\widehat{\mu^{u'}_s})(X^{u'}_s)^T\widehat{\tilde{a}_s}' \right] \\
    & +\widehat{\mathbb E}\left[D_\mu f(s,\widehat{X^{u'}_s},\widehat{\mu^{u'}_s})(X^{u'}_s)\right] ds + \nabla g(X^{u'}_T, \mu^{u'}_T)+ \widehat{\mathbb E} [D_\mu g(\widehat{X^{u'}_T}, \widehat{\mu^{u'}_T})(X^{u'}_T)]  \bigg|X^{u'}_t \Bigg]  \label{eq:lean_adjoint1} 
\end{align}
Combining \eqref{eq:u_diff} and \eqref{eq:lean_adjoint1}, using tower property, we have: 
\begin{align}
    \E[\tilde{a}_t'\big | X^{u'}_t] &= \E \Bigg[ \int^T_t \nabla \tilde b(s,X^{u'}_s,\mu^{u'}_s)^T \E[\tilde{a}'_s|X^{u'}_s]+\nabla \tilde f(s,X^{u'}_s,\mu^{u'}_s) +\widehat{\mathbb E}\left[D_\mu \tilde{b}(s,\widehat{X^{u'}_s},\widehat{\mu^{u'}_s})(X^{u'}_s)^T \widehat{\E}[\widehat{\tilde a_s'}|\widehat{X^{u'}_s}]\right]\\
    &+ \widehat{\mathbb E}\left[D_\mu \tilde f(s,\widehat{X^{u'}_s},\widehat{\mu^{u'}_s})(X^{u'}_s)\right] ds +\nabla g(X^{u'}_T, \mu^{u'}_T)+ \widehat{\mathbb E} [D_\mu g(\widehat{X^{u'}_T}, \widehat{\mu^{u'}_T})(X^{u'}_T)] \bigg| X^{u'}_t \Bigg] \label{eq:adjoint_equiv1}.
\end{align}
On the other hand for the basic adjoint matching case, by taking $v=u'$ in \eqref{eq:mfpa_adjoint} and take conditional expectation, we obtain: 
\begin{align}
    \E[a_t'\big | X^{u'}_t] &= \E \Bigg[ \int^T_t \nabla \tilde b(s,X^{u'}_s,\mu^{u'}_s)^T \E[a'_s|X^{u'}_s]+\nabla \tilde f(s,X^{u'}_s,\mu^{u'}_s) +\widehat{\mathbb E}\left[D_\mu \tilde{b}(s,\widehat{X^{u'}_s},\widehat{\mu^{u'}_s})(X^{u'}_s)^T \widehat{\E}[\widehat{a_s'}|\widehat{X^{u'}_s}]\right]\\
    &+ \widehat{\mathbb E}\left[D_\mu \tilde f(s,\widehat{X^{u'}_s},\widehat{\mu^{u'}_s})(X^{u'}_s)\right] ds +\nabla g(X^{u'}_T, \mu^{u'}_T)+ \widehat{\mathbb E} [D_\mu g(\widehat{X^{u'}_T}, \widehat{\mu^{u'}_T})(X^{u'}_T)] \bigg| X^{u'}_t \Bigg] \label{eq:adjoint_equiv2}.
\end{align}
Comparing \eqref{eq:adjoint_equiv1} and \eqref{eq:adjoint_equiv2}, since the previous lemma shows that the solution to \eqref{eq:adjoint_equiv1}/ \eqref{eq:adjoint_equiv2} is unique, we conclude that: 
\begin{align}
    \E[a_t'\big | X^{u'}_t] = \E[\tilde{a}_t'\big | X^{u'}_t] , \qquad \forall t \in [0,T]. 
\end{align}
Hence, the critical point of $ \E[\mcal{L}_{\rm L-MFAM}(w;u)]$, 
\begin{align}
   u'(t,X_t^{u'},\mu_t^{u'})= - \sigma_t^T \E[\tilde a (t;\mathbb{X}^{u'},u') \big | X^{u'}_t ]=- \sigma_t^T \E[a(t;\mathbb{X}^{u'},u') \big | X^{u'}_t ].
\end{align}
And clearly it is also a critical point to $\E[\mcal{L}_{\rm B-MFAM}(w;u)]|_{u=w}$:
\begin{align}
    \E[\mcal{L}_{\rm B\text{-}MFAM}(w;u)]|_{u=w}
    &:= \frac{1}{2}\E\left[ \int_0^T \left| w(t,X_t^{u},\mu_t^{u})
    +\sigma_t^Ta(t;\mathbb{X}^{u},u) \right|^2dt \right] \nonumber \\ 
    &= \frac{1}{2}\E\left[ \int_0^T \left| w(t,X_t^{u},\mu_t^{u})
    +\sigma_t^T \E[a(t;\mathbb{X}^{u},u)|X^u_t] \right|^2dt \right]\big|_{u=w} \\
    & +\frac{1}{2}\E\left[ \int_0^T \left| \sigma_t^T \big(\E[a(t;\mathbb{X}^{u},u)|X^u_t]- a(t;\mathbb{X}^{u},u)\big) \right|^2dt \right]\big|_{u=w} 
\end{align}
The conclusion then follows from Theorem~\ref{thm:mfam_optimality}. 
\end{proof}
We remark that the following Euler discretization \eqref{eq:lean_mf_adjoint_discrete} of equation \eqref{eq:lean_mf_adjoint}
\begin{align}
\widetilde a_n&=\widetilde a_{n+1}+\left[\nabla_x b_n^T\widetilde a_{n+1}+\nabla_x f_n+\widehat{\E}\left(D_\mu b_n^T\widehat{\widetilde a}_{n+1}+D_\mu f_n\right)\right]\Delta t,\\
    \widetilde a_N&=\nabla_x g(X_N,\mu_N)+\widehat{\E}\left[D_\mu g(\widehat X_N,\mu_N)(X_N)\right].\label{eq:lean_mf_adjoint_discrete}
\end{align}
is precisely the $Y$-recursion (see equation \eqref{d_mfbsde1}) of the sample-wise scheme in the current control-affine/quadratic setting. 

\section{RaNN basis and numerical stability}
In this section, we briefly discuss the formulation of function approximation via RaNN and discuss its implementation and stability. For simplicity, we consider a scalar output, i.e., $d_1=1$. At a fixed time point, the sample-based RaNN fitting problem can be written as
\begin{align}
    L(r)=\frac{1}{2}\|\Phi r-y\|_2^2,
    \qquad r\in\bR^L,\quad y\in\bR^M,\quad \Phi\in\bR^{M\times L}, \label{eq:no_ridge}
\end{align}
where $r$ denotes the trainable output weights and $\Phi$ is the feature matrix generated by the randomized basis functions $\{\phi_l\}_{l=1}^L$, with $\Phi_{ml}=\phi_l(x_m)$.

More specifically, let $\rho:\bR\rightarrow\bR$ be the activation function and
$\{(x_m,y_m)\}_{m=1}^M$, with $x_m\in\bR^d$ and $y_m\in\bR$, be the given data.
Let $W\in\bR^{L\times d}$ and $b\in\bR^L$ denote the randomly generated hidden-layer weights and biases respectively. Writing $w_l^\top$ for the $l$th row of $W$, the $l$th randomized basis function is
\begin{align}
    \phi_l(x)=\rho(w_l^\top x+b_l),
\end{align}
and hence
\begin{align}
    \Phi_{ml}=\phi_l(x_m)=\rho(w_l^\top x_m+b_l)=\rho\big((Wx_m+b)_l\big).
\end{align}

The first-order optimality condition for \eqref{eq:no_ridge} is
\begin{align}
\Phi^\top\Phi r=\Phi^\top y.
\end{align}
If $\Phi$ has full column rank, the least-squares coefficient is therefore
\begin{align}
r=(\Phi^\top\Phi)^{-1}\Phi^\top y.
\end{align}
This expression is only a theoretical representation; in numerical implementation, the inverse is not formed explicitly, and the corresponding linear system is solved directly.

Although the sample size is typically chosen such that $M\gg L$, this condition alone does not guarantee that $\Phi^\top\Phi$ is well-conditioned. In particular, different randomized features may be strongly correlated, causing the columns of $\Phi$ to become nearly linearly dependent. As a result, the Gram matrix $\Phi^\top\Phi$ may have a large condition number or may even become numerically singular. To improve stability, we employ ridge regression and consider
\begin{align}
L_\lambda(r)=\frac{1}{2}\|\Phi r-y\|_2^2+\frac{\lambda}{2}\|r\|_2^2,\qquad \lambda>0.
\end{align}
The corresponding first order condition and coefficient are
\begin{align}
(\Phi^\top\Phi+\lambda I)r=\Phi^\top y,\qquad r=(\Phi^\top\Phi+\lambda I)^{-1}\Phi^\top y.
\end{align}
To see that the ridge factor indeed helps improve stability overall, suppose that $\Phi^\top\Phi=V\Sigma^2V^\top$, where $\Sigma^2=\operatorname{diag}(\sigma_1^2,\ldots,\sigma_L^2)$ and $\sigma_1\geq\cdots\geq\sigma_L\geq0$. Then
\begin{align}
(\Phi^\top\Phi+\lambda I)^{-1}=V\operatorname{diag}\left(\frac{1}{\sigma_1^2+\lambda},\ldots,\frac{1}{\sigma_L^2+\lambda}\right)V^\top.
\end{align}
Thus, every eigenvalue of $\Phi^\top\Phi+\lambda I$ is bounded below by $\lambda$, and
\begin{align}
\|(\Phi^\top\Phi+\lambda I)^{-1}\|_2\leq\frac{1}{\lambda},\qquad \kappa_2(\Phi^\top\Phi+\lambda I)=\frac{\sigma_1^2+\lambda}{\sigma_L^2+\lambda}.
\end{align}
Consequently, ridge regularization prevents the regression system from becoming singular and reduces its sensitivity to small perturbations in the feature matrix and target values. However, an excessively small value of $\lambda$ may provide insufficient stabilization, whereas an excessively large value introduces additional regularization bias. In practice, $\lambda$ is selected empirically, for example from a small set of candidate values. 

In our experiments, since the optimal control is expected to vary regularly with both time and state, and the number of time steps $N$ can be large, we prefer to use a single global RaNN to approximate the control jointly in $(t,x)$. This avoids fitting independent approximations at every time step and allows information to be shared across neighboring time locations. For the global RaNN, a single coefficient vector is fitted jointly over all time locations. Let $\Phi_n\in\bR^{M\times L}$ and $y_n\in\bR^M$ denote the feature matrix and target vector at time $t_n$, respectively for $n=0,\ldots,N-1$. Define the stacked quantities
\begin{align}
\widetilde{\Phi}=(\Phi_0^\top,\ldots,\Phi_{N-1}^\top)^\top\in\bR^{NM\times L},\qquad \widetilde{y}=(y_0^\top,\ldots,y_{N-1}^\top)^\top\in\bR^{NM}.
\end{align}
In this case, the randomized features take both time and state as inputs. 
More precisely, let $W \in \bR^{L\times(d+1)}, b \in \bR^L$ and define $\bar{x}_{mn}=(t_n,x_{mn})$, then the $l$th randomized feature evaluated at $(t_n,x_{mn})$ is $\phi_l(t_n,x_{mn}):=\rho\big((W\bar{x}_{mn}+b)_l\big),$ so that
\begin{align}
    (\Phi_n)_{ml}=\phi_l(t_n,x_{mn})=\rho\big((W\bar{x}_{mn}+b)_l\big).
\end{align}
The global ridge-regression problem is similarly
\begin{align}
L_{\mathrm{global}}(\widetilde r)=\frac{1}{2}\|\widetilde{\Phi}\widetilde r-\widetilde y\|_2^2+\frac{\lambda}{2}\|\widetilde r\|_2^2.
\end{align}
Equivalently, using the block structure across time, we have: 
\begin{align}
L_{\mathrm{global}}(\widetilde r)=\frac{1}{2}\sum_{n=0}^{N-1}\|\Phi_n\widetilde r-y_n\|_2^2+\frac{\lambda}{2}\|\widetilde r\|_2^2.
\end{align}
The associated first order condition is then
\begin{align}
\left(\sum_{n=0}^{N-1}\Phi_n^\top\Phi_n+\lambda I\right)\widetilde r=\sum_{n=0}^{N-1}\Phi_n^\top y_n,
\end{align}
and hence
\begin{align}
\widetilde r=\left(\sum_{n=0}^{N-1}\Phi_n^\top\Phi_n+\lambda I\right)^{-1}\left(\sum_{n=0}^{N-1}\Phi_n^\top y_n\right).
\end{align}
Therefore, the global RaNN does not average the separately fitted timewise coefficients. Instead, it solves one pooled regression problem using all $NM$ time--state samples. This pooling can improve the effective information content of the regression system, although its conditioning still depends on the correlation structure of the global randomized features.

\section{Numerical examples}
In this section, we present five numerical examples to illustrate the effectiveness, flexibility, and practical potential of the proposed approach. The first example considers a standard stochastic optimal control problem without mean-field interaction. Its purpose is to benchmark the proposed method against a direct neural-network approach, while also examining the conditioning of the regression problem and the sensitivity of the RaNN approximation to the number of hidden features. The second example demonstrates that the proposed algorithm can effectively handle mean-field control problems with controlled diffusion, including settings in which the distribution of the control enters the dynamics or objective. The third example provides a toy fine-tuning problem to illustrate the use of the method for distributional adaptation. The fourth example considers a problem with constraints imposed not only on the terminal distribution but also on intermediate-time marginals, showing that the proposed approach can capture the prescribed pathwise distributional evolution while still matching the terminal target. Finally, the fifth example demonstrates the scalability and practical applicability of the method through an unpaired facial-expression transport task on the CelebA dataset.
\subsection{Example 1: A linear quadratic SOCP}
The loss functional is given by
\begin{align}
J(u)=\mathbb{E}[\frac{1}{2} \int^T_0 ( X^T Q_t X_t +  u^T_t R_t u_t ) dt + \frac{1}{2} X^T_T S X_T]
\end{align}
and $X_t$ follows the stochastic process 
\begin{align}
dX_t= (A_t X_t + B_t u_t) dt + C_t dW_t, && X_0 =x 
\end{align}
where $X \in \bR^d$ and $W \in \bR^m$. 

For this example, we take $m=d=20$ and $A_t=1.0 I, B_t=C_t=I, r_t=2.0, q_t=2.0 I, s_t=I$. This set of parameters are reported to be the `hard' setup according to \cite{jiequn2} which means that it is in general harder for the vanilla deep learning algorithms to produce control that converges to the true optimal control. 

Since all coefficients are multiples of the identity, the Riccati equation is found to take the following form: with \(P_t=p_tI\):
\[
\dot p_t = -q_t - 2 p_t + \frac{1}{2} p_t^2,
\qquad p_T = s_T.
\]
Solving the Riccati ODE gives: 
\begin{align}
    p_t&=\frac{p^1 -\gamma e^{2\sqrt{2}(t-1)}p^2}{1-\gamma e^{2\sqrt{2}(t-1)}}, &&p^1=2+2\sqrt{2}, \ p^2=2-2\sqrt{2}, \gamma=\frac{1-p^1}{1-p^2}. 
\end{align}
and the optimal control is the given by $u^*_t = -\frac{1}{2}P_t X_t.$
We solve the problem by using the proposed algorithm and benchmark it against the most effective approach based on deep learning proposed in \cite{jiequn1}. The learned controls are evaluated using
\[
\operatorname{MSE}_{u} =
\Delta t\sum_{n=0}^{N-1} \mathbb{E}\!\left[ \left| \widehat{u}_n(X_n^{*})-u_n^{*}(X_n^{*}) \right|^2 \right],
\]
where $u_n^{*}$ and $X_n^{*}$ denote the exact Riccati control and the corresponding optimally controlled state process. The test specifications are listed in Table \ref{tab:lq_base_configuration}. Importantly, for the global SMP--RaNN method, we used the third order polynomial encoding to create more input features.
Since, only one global RaNN is used for fitting the entire control, the number is hidden layers is intentionally taken to be large (320). In Figure \ref{fig:eg1_orginal_comparison}, comparison between the three methods show that SMP-RaNN based methods shows convergence within only 10 iterations and achieves very low MSE compared to the Direct approach. In the figure on the right, it again shows that the SMP-RaNN outputs Direct NN in terms of time efficiency in this example. In particular, the time-wise RaNN achieves low MSE in less than on second while the global RaNN in this example shows a smoother and milder convergence curve. 

\begin{table}[htbp]
\centering
\caption{Base configuration for the comparison of the timewise SMP--RaNN, global SMP--RaNN, and direct NN methods.}
\label{tab:lq_base_configuration}
\small
\renewcommand{\arraystretch}{1.15}
\setlength{\tabcolsep}{5pt}
\resizebox{\textwidth}{!}{
\begin{tabular}{lccc}
\toprule
\textbf{Parameter} & \textbf{Timewise SMP--RaNN} & \textbf{Global SMP--RaNN} & \textbf{Direct NN} \\
\midrule
Common setup & \multicolumn{3}{c}{$d=20$, $T=1$, $N=20$, $M=8000$, $M_{\mathrm{eval}}=4000$, $X_0\sim\mathcal{N}(0,0.5^2I_d)$} \\
LQ coefficients & \multicolumn{3}{c}{$a=b=\sigma=g=1$, $q=r=2$} \\
Control form & $\{u_n(x)\}_{n=0}^{N-1}$ & $u(t,x)$ & $\{u_n(x)\}_{n=0}^{N-1}$ \\
Number of controls & $20$ & $1$ & $20$ \\
Input dimension & $20$ & $1+(3+1)d=81$ & $20$ \\
Hidden structure & $256$ random features & $320$ random features & two layers $(128,128)$ \\
Architecture & $20$--$256$--$20$ & $81$--$320$--$20$ & $20$--$128$--$128$--$20$ \\
Activation & $\tanh$ & $\tanh$ & $\tanh$ \\
Training method & timewise ridge projection & pooled ridge projection & Adam backpropagation \\
Ridge parameter & $2\times10^{-3}$ & $2\times10^{-3}$ & -- \\
Step size & $0.4(k+1)^{-1/2}$ & $0.1(k+1)^{-1/2}$ & learning rate $8\times10^{-3}$ \\
Iterations & $100$ & $100$ & $2500$ \\
Time--state encoding (Global RaNN) & State $x$ & $\mathcal{E}_3(t,x)=(\widetilde{t},x,\widetilde{t}x,\widetilde{t}^{\,2}x,\widetilde{t}^{\,3}x)\in\mathbb{R}^{81}$, $\widetilde{t}=2t/T-1$ & State $x$ \\
\bottomrule
\end{tabular}}
\end{table}

\begin{figure}[h]
    \centering 
    \includegraphics[width=1.0\textwidth]{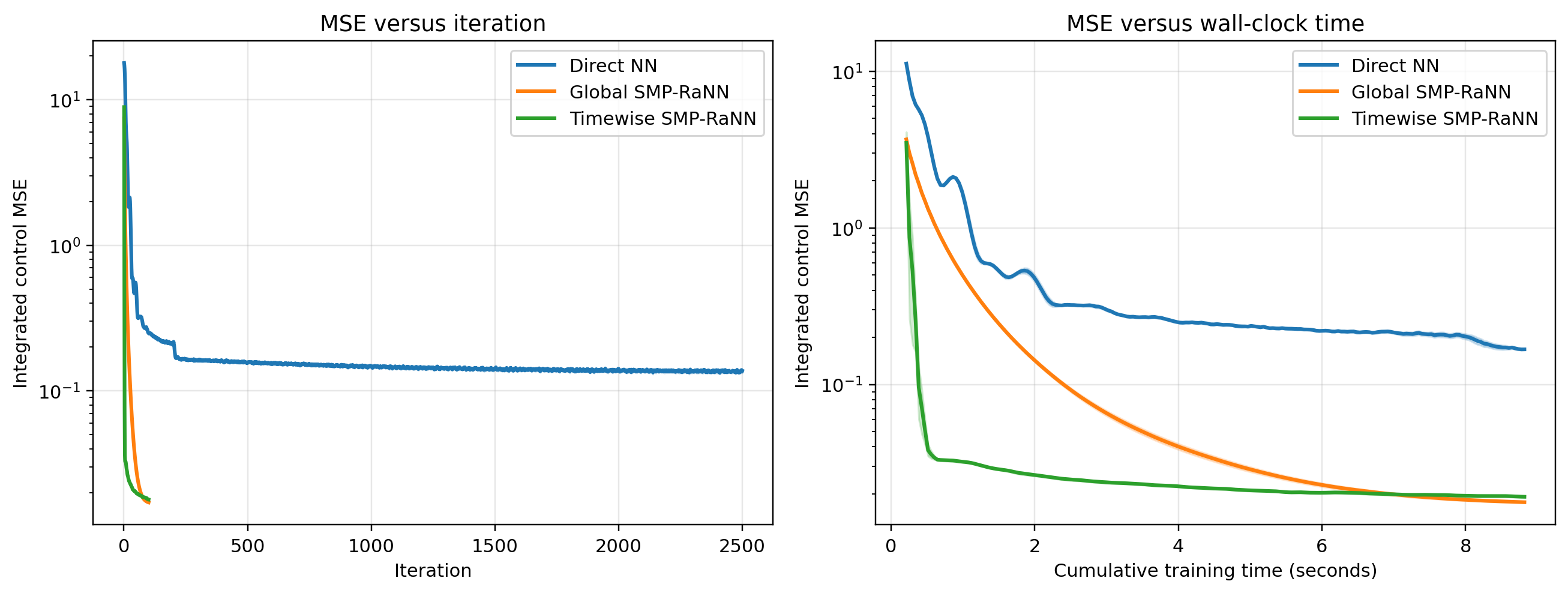} 
    \caption{Comparing among the time-wise SMP-RaNN method, the global SMP-RaNN and the direct NN method. Left: Control MSE over number of training iterations. Right: Error decay of MSE over wall-clock training time.}
    \label{fig:eg1_orginal_comparison} 
\end{figure}

\subsubsection*{Stability and Sensitivity test for example 1}
In this section, we perform ablation test based on choosing different feature size (hidden dimension), the training path and the number of the temporal discretizations. The impact of those choices on the final control RMSE is presented and analyzed in Figure~\ref{fig:eg1_ablation}. The test specifications is kept the same as those in Table~\ref{tab:lq_base_configuration} while only changing one parameter at a time. The test specifications are detailed in Figure~\ref{fig:eg1_ablation}.

The sample-size experiment shows that the final control error decreases as the number of training paths ($M$) increases, with the global SMP--RaNN attaining a consistently smaller mean error. Increasing the number of time steps also substantially reduces the final error for both methods. In the hidden-basis experiment, the global method improves markedly as the hidden dimension increases, whereas the timewise method becomes numerically unstable at $H=512$; the corresponding statistics retain all completed unstable runs and therefore reflect the ill-conditioning of the regression system.

\begin{figure}[h]
    \centering 
    \includegraphics[width=1.0\textwidth]{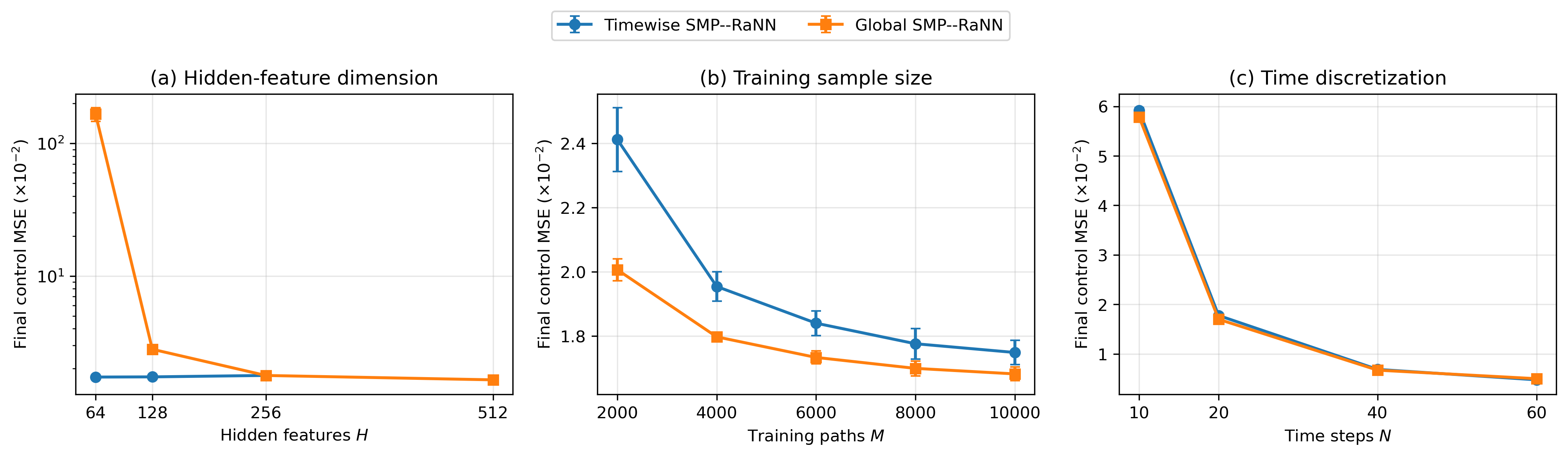} 
    \caption{Final integrated control MSE after 100 iterations under three parameter sweeps; Error bars show the sample standard deviation. (a) Comparison with respect to the number $H$ of randomized hidden features, using a logarithmic vertical axis because of the comparatively large global-RaNN error at $H=64$. The case $H=512$ is omitted for the timewise method because it becomes numerically unstable, producing an MSE of $(2.67\pm5.10)\times10^{19}$ from the five completed runs. (b) Comparison with respect to the number $M$ of training paths. (c) Comparison with respect to the number $N$ of time-discretization steps.}
    \label{fig:eg1_ablation} 
\end{figure}

Meanwhile, we examine the condition numbers of the Gram matrix $\Phi^\top\Phi$ (raw) and the regularized Gram matrix $\Phi^\top\Phi+\lambda I$ (ridge) arising in the RaNN projection; see Figure~\ref{fig:eg1_condition_number}. In both subfigures, the raw Gram matrices associated with the timewise SMP--RaNN are numerically singular, and their condition numbers are therefore reported as $+\infty$ and they are capped by finite numbers for plotting. This highlights the importance of ridge regularization in stabilizing the regression systems. By contrast, the global SMP--RaNN exhibits consistently smaller condition numbers, suggesting that the time--state formulation leads to a better-conditioned feature system.
\begin{figure}[h]
    \centering 
    \includegraphics[width=0.95\textwidth]{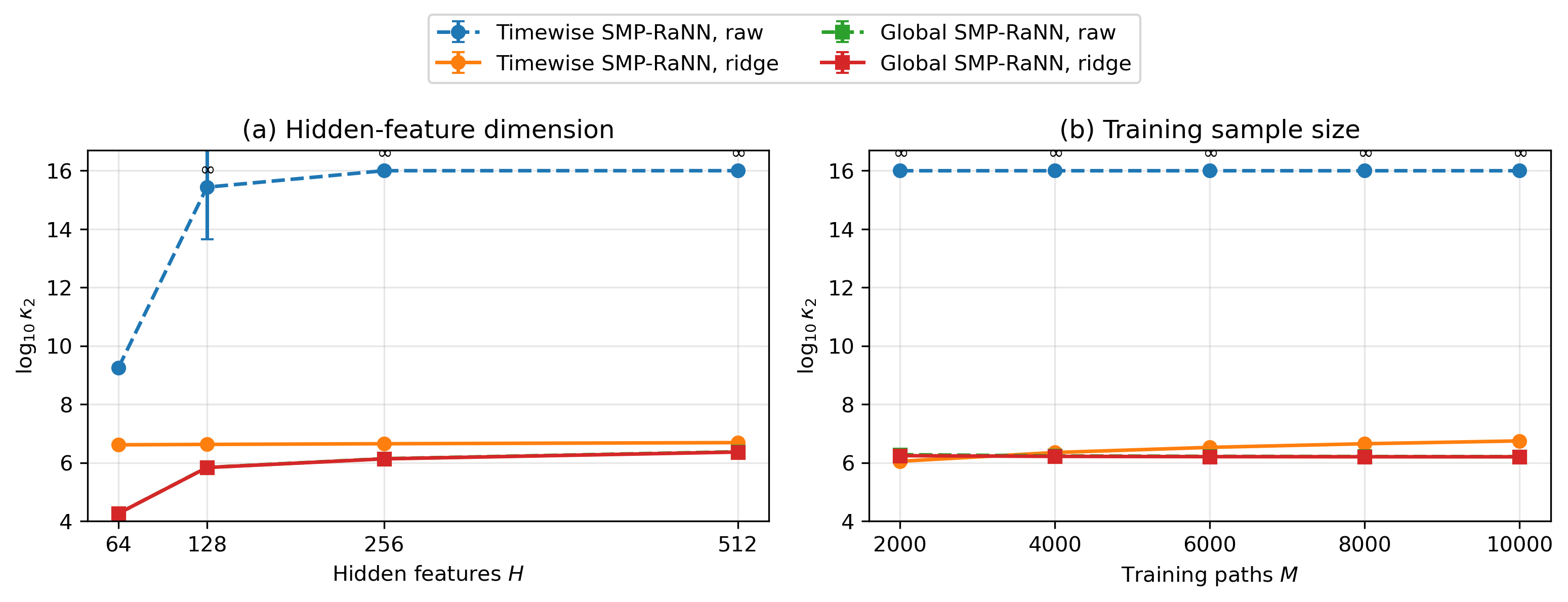} 
    \caption{Comparison of condition number between the four different gram matrices post 100 iterations training. Condition numbers are in log scale due to their large values. (a) Condition number versus number of hidden features $H$. (b) Condition number versus the size of training paths $M$.}
    \label{fig:eg1_condition_number} 
\end{figure}
\subsection{Example 2: Mean field SOC.}
\subsubsection*{Mean-variance portfolio optimization}
This example is related to a problem with controlled diffusion and it is a standard mean-variance portfolio investing problem: 
\begin{align}
    J(\alpha)=\frac{\eta}{2}\E[X^2_T]-\frac{\eta}{2}(\E[X_T])^2 -\E[X_T] && \text{mean-variance loss}
\end{align}
where $X_t$ follows the following dynamics: 
\begin{align}
    dX_t=(r X_t + \rho u_t  ) dt + \theta u_t  dB_t, && X_0 \sim \mu_0.
\end{align}
As such, the Hamiltonian takes the following form:  
\[
H=(rx+\rho u)Y+\theta u Z 
\]
The following adjoint process is then consructed accordingly: 
\begin{align}
    dY_t=-rY_t dt + Z_t dW_t, && Y_T=\eta X_T-\eta \Tilde{\E}[\tilde{X}_T]-1.
\end{align}
Then we also have: 
\begin{align}
    J'_u=\nabla_u H = \rho Y + \theta Z
\end{align}
in which case the diffusion is also controlled, making it a more challenging problem numerically.

For this example, to benchmark the result in \cite{pham2} we take $T=1$, $x_0=0.1$, $r=0$, $\rho=0.1$, $\theta=0.4$ and $\eta=1.0$.
The exact control function is given by: 
\begin{align}
  u^*(X_t,\mathbb{P}_{X_t}) = -\frac{\rho}{\theta^2}\Big(X_t - \E[X_t] -\frac{1}{\eta} \exp\left((\frac{\rho^2}{\theta^2}-r) (T-t)\right)\Big).
\end{align}
We perform our test for $10$ individual runs and compare the performance (control RMSE) of Algorithm \ref{algorithm_main2} and the DirectNN method proposed in \cite{Carmona4}. The control RMSE is defined as: 
\begin{align}
\mathrm{RMSE} := \left(\frac{1}{N M}\sum_{n=0}^{N-1} \sum_{i=1}^{M} \left|\widehat{u}_n\left(X_n^{*,i}\right) -u_n^*\left(X_n^{*,i}\right) \right|^2 \right)^{1/2}. 
\label{eq:control_rmse}
\end{align}
where $\widehat{u}$ denotes the learned control, $u^*$ the optimal control, and $X^*_n$ the state corresponding to the optimal control. For control approximation in Algorithm \ref{algorithm_main2}, we again adopt two approaches, namely approximating each $u_t(\cdot)$ individually by a RaNN and approximating $u(\cdot)$ globally with $(t,x)$ as arguments but with feature encoding.  
\begin{figure}[h]
    \centering 
    \includegraphics[width=0.9\textwidth]{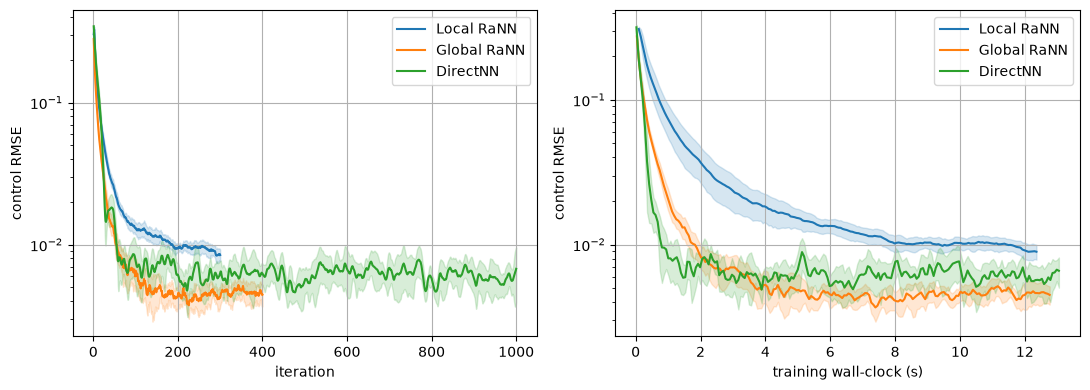} 
    \caption{Comparison of RMSE between the local RaNN, Global RaNN and DirectNN method for the mean-variance portfolio control problem.Left: RMSE with respect to number of iterations. Right: RMSE with respect to wall-clock time.}
    \label{fig:mean_field_res1} 
\end{figure}
\begin{figure}[h]
    \centering 
    \includegraphics[width=0.9\textwidth]{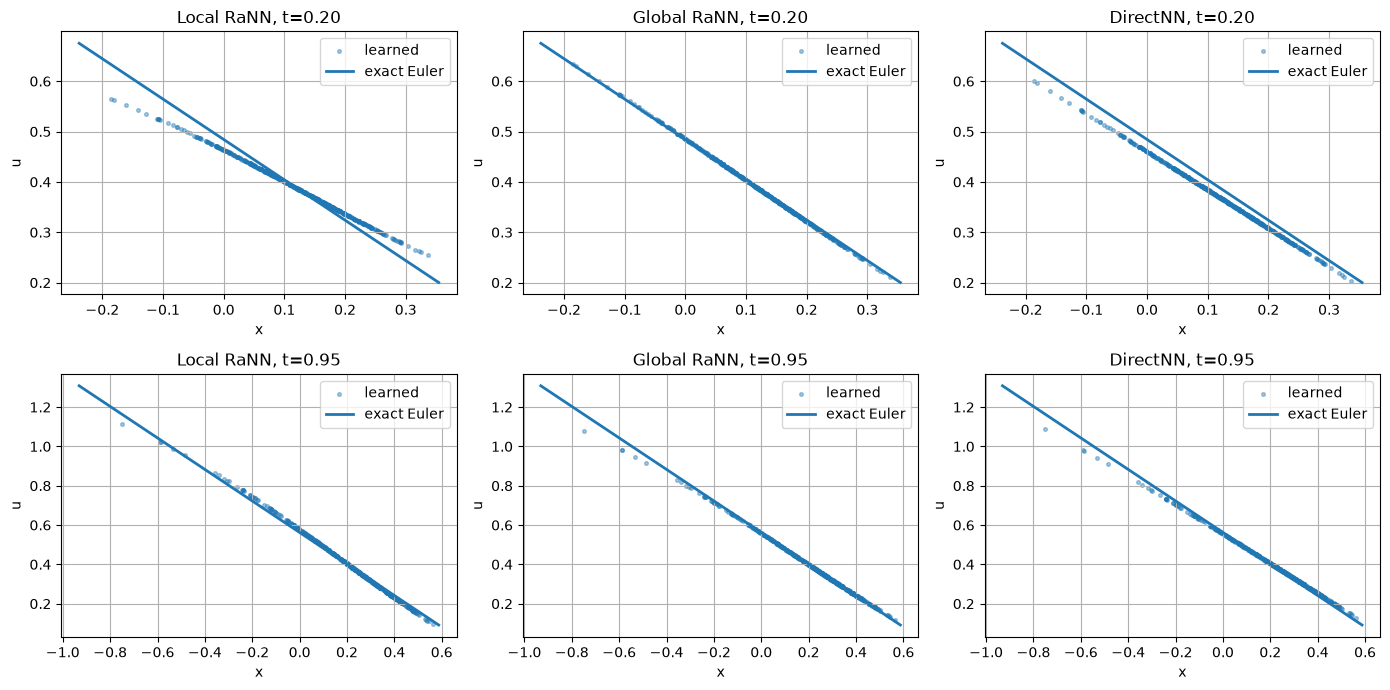} 
    \caption{Comparison between the local RaNN, Global RaNN and DirectNN method for the mean-variance portfolio control problem. }
    \label{fig:mean_field_res2} 
\end{figure}
It is observed from Figure \ref{fig:mean_field_res1} that compared to the standard DirectNN, the local RaNN approach is less time-efficient and within the same wall-clock time, it does not achieve better accuracy than benchmark. the Global RaNN on the other hand, achieves better accuracy in less iterations within comparable wall-clock time. One reason behind this phenomenon is that the control is usually continuous in $(t,x)$ variable and using one large RaNN to approximate the control is both cheaper and more effective. This observation is also confirmed in Figure \ref{fig:mean_field_res2}. At time $t=0.2$, the control learned based on the global RaNN approach shows highest accuracy. 

To further test the solution, we also compute the value function to related to this MFC problem, which is also given in \cite{pham2},
\begin{align}
    v(t,\mu) &= \frac{\eta}{2}e^{\left(2r-\frac{\rho^2}{\theta^2}\right)(T-t)}
\text{Var}_{\mu}(X) -e^{r(T-t)}\E_{\mu}[X] -\frac{1}{2\eta}
\left( e^{\frac{\rho^2}{\theta^2}(T-t)}-1 \right).
\end{align}
Define $\kappa=\frac{\rho^2}{\theta^2}$,At $t=0$, $v(\mu_0)=\frac{\eta}{2}e^{(2r-\kappa)T}\text{Var}(X_0)-e^{rT}\E[X_0]-\frac{e^{\kappa T}-1}{2\eta}.$

To further ensure the reliability and accuracy of our method, we test the algorithm over six different initial distributions with the following model parameters $T=1,r=0.2,\rho=0.2, \theta=0.5, \eta=2$:
\begin{enumerate}
    \item $X_0\sim \mcal{N}(0.1,0.04)$.
    \item $X_0\sim \mcal{N}(0.2,0.025^2)$.
    \item $X_0\sim \mcal{N}(0.3,0.025^2)$.
    \item $X_0=0.5(-\frac{\sqrt{3}}{10}+0.1+0.1\xi_1)+0.5(-\frac{\sqrt{3}}{10}+0.1+0.1\xi_2)$, \quad $\xi_1, \xi_2 \sim \textbf{Uni}(0,2\sqrt{3})$, i.i.d.
    \item $X_0=0.5(-0.1+0.05+0.1\xi_1)+0.5(-0.1+0.05+0.1\xi_2)$, \quad  $\xi_1, \xi_2 \sim \textbf{Exp}(1)$, i.i.d.  
    \item $X_0=a+[-\mathbf{1}_{5U<2}k + \mathbf{1}_{5U>3} k] + \theta Y$ with $U \sim U(0,1), a=0.2, k=0.3, \theta=0.07, Y \sim\mcal{N}(0,1).$
\end{enumerate}
In Table \ref{tab:mean_variance_value_comparison}, we compare the computed value function evaluated at $t=0$ based on our learned control and the analytic value for the six different distribution set up. To benchmark against the result under the discretization error and MC error, we also compute the value function via Euler MC under the analytic control function. For all the tests below, we used the number of temporal discretization $N=80$ with width (number of randomized hidden features/basis) equal to $128$. It shows that the computed value function is very close to that of the analytic result.  

\begin{table}[htbp]
\centering
\caption{Comparison of the exact value, Global RaNN approximation, and Euler Monte Carlo approximation over 10 independent runs. Results are reported as mean $\pm$ standard error.}
\label{tab:mean_variance_value_comparison}
\begin{tabular}{lccc}
\hline
Initial law & Exact value & Global RaNN & Euler MC \\
\hline
$\mathcal{N}(0.1,0.04)$
& $-0.115$
& $-0.114 \pm 1.40\times 10^{-3}$
& $-0.116 \pm 1.44\times 10^{-3}$ \\

$\mathcal{N}(0.2,0.025^2)$
& $-0.287$
& $-0.287 \pm 5.77\times 10^{-4}$
& $-0.287 \pm 5.46\times 10^{-4}$ \\

$\mathcal{N}(0.3,0.025^2)$
& $-0.409$
& $-0.408 \pm 5.80\times 10^{-4}$
& $-0.408 \pm 9.16\times 10^{-4}$ \\

Averaged uniform law
& $-0.159$
& $-0.158 \pm 7.58\times 10^{-4}$
& $-0.160 \pm 6.51\times 10^{-4}$ \\

Averaged exponential law
& $-0.098$
& $-0.099 \pm 7.22\times 10^{-4}$
& $-0.099 \pm 7.36\times 10^{-4}$ \\

Three-component Gaussian mixture
& $-0.190$
& $-0.190 \pm 1.67\times 10^{-3}$
& $-0.193 \pm 1.45\times 10^{-3}$ \\
\hline
\end{tabular}
\end{table}

\subsubsection*{Optimal liquidity with price-impact}
In this example, we study a price impact model where the interactions take place via the controls \cite{carmona2015probabilistic,carmona2018probabilistic1}. 
We denote an inifinitesimal trader's inventory at time $t$ by $X_t$ which is assumed to evolve under the following SDE: 
$$dX_t = \alpha_t dt + \sigma dW_t, \quad X_0 \sim \mu_0.$$
Denoting by $\nu_t^{\alpha}$ the law of the control at time $t$, the cost is given by
\begin{align}
    J(\alpha) = \mathbb{E}\left[\int_{0}^{T} \underbrace{\left( \frac{c_{\alpha}}{2} \alpha_{t}^{2} + \frac{c_{X}}{2} X_{t}^{2} - \gamma X_{t} \int_{\mathbb{R}} a \nu_{t}^{\alpha}(a) \,da \right)}_{f(X_t,\alpha_t, \mcal{L}_{(X_t, \alpha_t)})}  \, dt + \frac{c_{g}}{2}X_{T}^{2}\right] 
\end{align}
where $\gamma, c_{\alpha}, c_{X}, c_{g}$ are constants. Following the derivation of \cite{Carmona5}, we construct the following Hamiltonian: 
\begin{align}
    H(t,X_t,Y_t,Z_t, \alpha_t, \mcal{L}_{X_t, \alpha_t})= \alpha_t Y_t + \sigma Z_t + f(X_t, \alpha_t,  \mcal{L}_{(X_t, \alpha_t)})
\end{align}
where the adjoint process $(Y_t, Z_t)$ takes the form (note that $H$ does not depend explicitly on the distribution of $X_t$):
\begin{align}
    dY_t = - (c_X X_t -\gamma \E[\alpha_t])dt+ Z_t dW_t, \quad Y_T = c_g X_T.
\end{align}
Importantly, since the running cost also depends on the distribution of the control, the `generalized gradient' takes the form: 
\begin{align}
    J'_u&=\nabla_u H + \tilde{\E}[\partial_\nu H(\tilde{X_t}, \tilde{Y}_t,\tilde{Z}_t, \tilde{\alpha}_t, \mcal{L}_{(\tilde{X}_t, \tilde{\alpha}_t)})](X_t, \alpha_t) \nonumber\\ 
    &= Y_t + c_\alpha \alpha_t - \gamma \tilde{\E}[\tilde{X}_t].
\end{align}
In this example, we take $(c_\alpha, c_X, c_g, \gamma, \sigma, T)=(1,2,0.3,1,0.5, 1)$, and the initial distribution $\mu_0 \sim \mcal{N}(5,0.3)$. The three model specifications are also summarized in Table \ref{tab:price_impact}.
\begin{table}[htbp]
\centering
\caption{Hyperparameter specification for the three methods.}
\label{tab:price_impact}
\begin{tabular}{lccc}
\toprule
 & Local RaNN & Global RaNN & Direct NN \\
\midrule
Number of iterations 
    & $200$ & $200$ & $1500$ \\
Hidden width 
    & $128$ & $128$ & $128$ \\
Number of networks 
    & $N=50$ & $1$ & $1$ \\
Initial learning rate 
    & $0.6$ & $0.6$ & $0.1$ \\
Learning-rate decay 
    & $(k+1)^{-0.5}$ & $(k+1)^{-0.6}$ & $(k+1)^{-0.3}$ \\
Optimizer/update 
    & Gradient + ridge projection 
    & Gradient + ridge projection 
    & Adam \\
Ridge parameter 
    & $10^{-5}$ & $10^{-5}$ & -- \\
Activation 
    & $\tanh$ & $\tanh$ & $\tanh$ \\
\bottomrule
\end{tabular}
\end{table}

The optimal control is given by
\begin{align}
u_t^*= -\frac{P_t}{c_\alpha}\left(X_t-m_t\right) -\frac{S_t}{c_\alpha}m_t,
    \qquad m_t:=\E[X_t],
\end{align}
where $P_t$ and $S_t$ solve the Riccati equations
\begin{align}
\dot P_t&=\frac{P_t^2}{c_\alpha}-c_X, \qquad P_T=c_g,\\
\dot S_t&=\frac{S_t^2}{c_\alpha}-c_X, \qquad S_T=c_g-\gamma.
\end{align}
And the mean state $m_t$ satisfies
\begin{align}
\dot m_t =-\frac{S_t}{c_\alpha}m_t, \qquad m_0=\int x\,\mu_0(dx).
\end{align}
It is observed from Figure \ref{fig:price_impact} that both the local RaNN and the global RaNN method reaches high accuracy within only 200 iterations. According to Table \ref{tab:price_impact}, under the similar test setup, our algorithm based on SMP also achieves higher accuracy with less wall-clock time, demonstrating its efficiency specifically for this example.  
\begin{figure}[h]
    \centering 
    \includegraphics[width=0.9\textwidth]{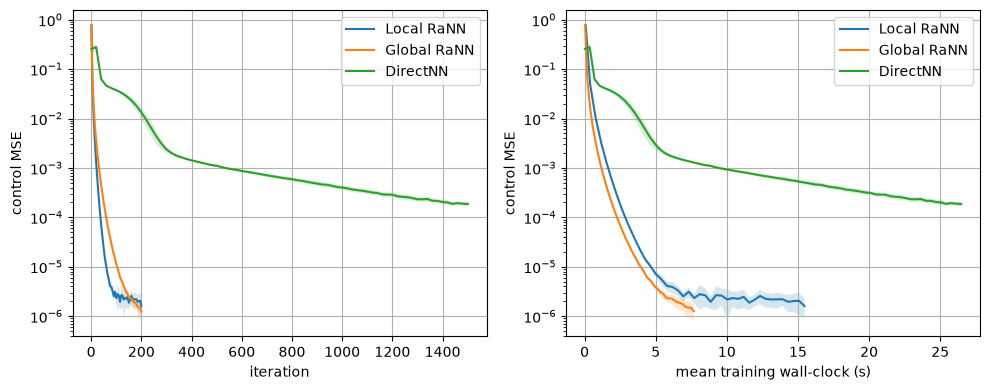} 
    \caption{The optimal liquidity model with price-impact: Comparison of control RMSE between the local RaNN, Global RaNN and DirectNN method for the optimal liquidity with price-impact problem. Left: RMSE with respect to number of iterations. Right: RMSE with respect to wall-clock time.}
    \label{fig:price_impact} 
\end{figure}

\subsection{Example 3: A fine tuning example (MFC)}
In this example, we perform the fine tuning task. The goal is to transport the base distribution
\begin{align}
    p_{\text{base}} = \frac{1}{2} \phi(x; m_1, \sigma^2_1) + \frac{1}{2} \phi(x; m_2, \sigma^2_2)
\end{align}
towards a target distribution 
\begin{align}
    p_{\text{targe}} \propto \exp(r(x))p_{\text{base}}
\end{align}
where in the setup above $r(x)=\alpha x$, $m_1=-2, m_2=2$ and $\sigma_1=0.55$ and $\sigma_2=0.55$ and $\phi(x;m,\sigma)= \frac{1}{\sqrt{2 \pi \sigma^2}}\exp(-\frac{1}{2\sigma^2}(x-m)^2)$. In this example, we take $r(x)=-0.42x$. 

According to the Langevin dynamics, the distribution of the terminal state $X_T$ for the following SDE follows the base distribution:
\begin{align}
   dX_t = b(t,X_t) dt + \sigma dW_t, X_0 \sim p_{\text{base}} 
\end{align}
with $b(t,x)=\frac{\sigma^2}{2} \partial_x \log p_{\text{base}}(x)$. And in this example, $b(x)=\frac{\sigma^2}{2a^2}\left( 2 \tanh(\frac{2x}{a^2})-x\right)$, $a=0.55$. In our fine tuning setup, the goal is to minimize the following: 
\begin{align}
    J(u)= \E [\frac{1}{2}\int^1_0 u^2(t,X_t) dt] + \lambda_g \frac{1}{2} D(\mu_T, \nu). 
\end{align}
over the admissible controls and the underlying process is defined as 
\begin{align}
   dX_t = \left( b(t,X_t) + \sigma u_t \right)dt + \sigma dW_t, \quad X_0 \sim p_{\text{base}}. 
\end{align}
In this toy example, we simply assume that the base distribution is known and so that the score of $p_{\rm base}$ can be computed analytically. When $p_{\text{base}}$ is only accessible in the form of data, such quantity can be learned by the well-established score-matching methods. 
Note that if we take the terminal loss to be KL-divergence, it can be written equivalently as
\begin{align}
D(\mu_T, \nu) &=\int_{\mathbb{R}^d}\log\left(\frac{\rho_T(x)}{\nu(x)}\right)\rho_T(x)\,dx \\
&=\text{KL}(\mu_T\,\Vert\,\rho_{\rm base}) +\E[V(X_T)]+\log Z,
\end{align}
and $Z=\int_{\mathbb{R}^d}e^{-V(x)}\rho_{\rm base}(x)\,dx$ with $V(x)=-r(x)$ which is a constant. 
Consequently, the Lions derivative appearing in the terminal adjoint is
\begin{align}
    D_\mu D(\mu_T, \nu)(x)=\nabla\log\rho_T(x)-\nabla\log\rho_{\rm base}(x)+\nabla V(x).
\end{align}
Numerically, we estimate it from an independent terminal particle cloud $\{X_T^i\}_{i=1}^M$ using a Gaussian kernel density estimator
\begin{align}
\widehat{\rho}_T(x)= \frac{1}{M}\sum_{i=1}^M K_h(x-X_T^i),
\qquad K_h(z)= \frac{1}{(2\pi h^2)^{d/2}} \exp\left(-\frac{|z|^2}{2h^2}\right).
\end{align}
Its score is computed directly by
\begin{align}
\nabla\log\widehat{\rho}_T(x) &=-\frac{1}{h^2}\frac{\sum_{i=1}^M (x-X_T^i)K_h(x-X_T^i)}{\sum_{i=1}^M K_h(x-X_T^i)}.
\end{align}
 
We first test our training result by changing the diffusion $\sigma$. In \ref{fig:fine_tune1D_diffusion}, the first row shows the generated result using our previous descent algorithm while the second row is obtained via the adjoint matching algorithm. It is noted that both algorithm completes the fine-tuning task well with the descent algorithm achieving slightly lower $W_2$ in this case under both choices of $\sigma$. In training, we initialize the control using a short Gibbs-informed warm start constructed solely from the base distribution and the potential $V$, with importance weights proportional to $e^{-V}$.

\begin{figure}[h]
    \centering 
    \includegraphics[width=0.8\textwidth]{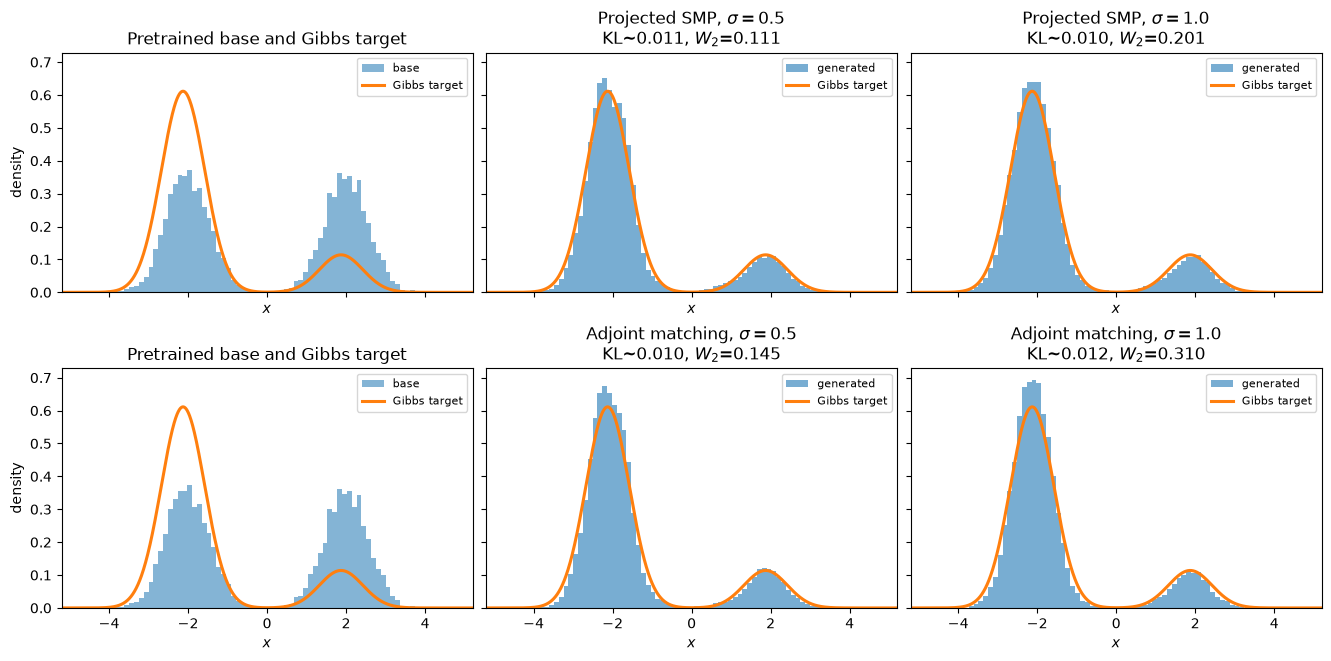} 
    \caption{Fine-tuning the drift of base SDE based on the descent and MFC adjoint-matching algorithm with $\lambda_g=200$ under different diffusion setup.}
    \label{fig:fine_tune1D_diffusion} 
\end{figure}
To further study the impact of different choice of the penalty factor, we perform study over $\lambda_g=20,80,160,320,1280$. The experiment is repeated five individual times. The mean terminal $\rm KL/W_2$ is shown with respect to $\lambda_g$ in \ref{fig:fine_tune1D_lambda} with error bar plotted. Also, in Figure \ref{fig:fine_tune1D_lambda_distribution}, the different final distributions are demonstrated for each $\lambda_g$. From those, we see while $KL$ divergence shows mild fluctuation in its mean, $W_2$ changes more noticeably. This is because $\rm KL$ is sensitive to weight mismatch whereas $W_2$ is sensitive to where this mismatch is located. Thus, it is possible that a small KL value could have larger $W_2$ loss, hence using only $\rm KL$ may not fully preserve the geometry in the distribution. From our test, it is shown that it is more reliable to keep $\lambda_g$ relative small for this problem: large $\lambda_g$ could lead to stability issue by amplifying the approximation noise. 
\begin{figure}[h]
    \centering 
    \includegraphics[width=0.8\textwidth]{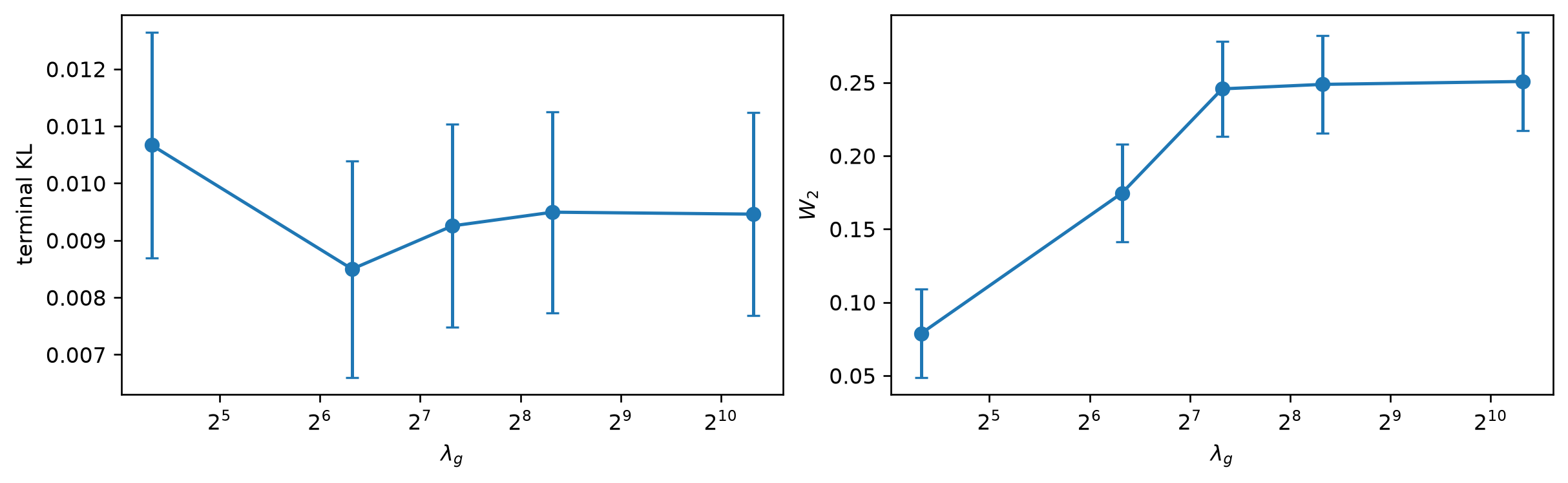} 
    \caption{The change of $KL/W_2$ with respect to the penalty factor $\lambda_g$. }
    \label{fig:fine_tune1D_lambda} 
\end{figure}
\begin{figure}[h]
    \centering 
    \includegraphics[width=1.0\textwidth]{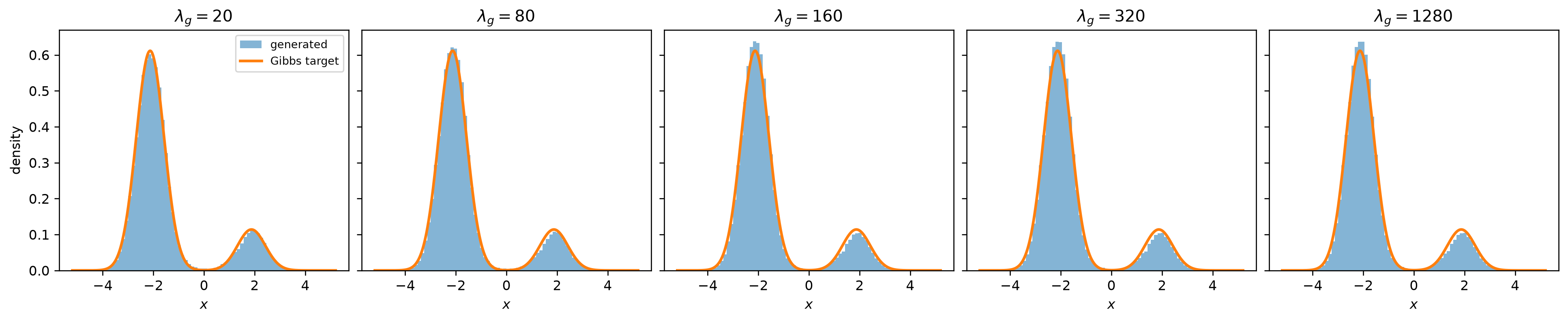} 
    \caption{Comparison between generated empirical distributions under different penalty factors $\lambda_g=20,80,160,320,1280$.}
    \label{fig:fine_tune1D_lambda_distribution} 
\end{figure}

\subsection{Example 4:2-D endpoint law transport with path supervision.}
In this example, we consider the two-dimensional Gaussian marginal flow studied in \cite{pham5}, whose time-dependent mean follows a parabolic trajectory while the covariance matrix remains constant:
\begin{align}
    m_t=(-2.25+4.5t,\,2.75\times4t(1-t)),\qquad \rho_t=\mcal{N}\left(m_t,\operatorname{diag}(0.35^2,0.12^2)\right).
\end{align}
We set $\lambda_g=140$ and $\lambda_f=1400$, and use $N=100$ time steps. Under this formulation, the terminal distribution is penalized against the terminal target, while the intermediate marginals are additionally supervised along the prescribed distributional path. Due to the relatively fine temporal discretization, we employ a global RaNN that takes both time and state as inputs, with Fourier encoding applied to the inputs. The model is trained for $600$ outer iterations, with $5$ inner fitting steps performed at each iteration. We take the diffusion coefficient to be $\sigma=0.15$ and use a batch size of $M=3000$ particles in each outer iteration, resulting in $3000\times N$ $(\text{time},\text{state})$ training pairs. In each inner fitting step, $8192$ samples are used for regression.

The results are reported in Table \ref{tab:mfam_w2_comparison}. As shown by the first column, incorporating intermediate marginal supervision does not necessarily improve the accuracy of the terminal distribution matching; in this example, the terminal-only formulation achieves a smaller terminal $W_2$ error. In contrast, the remaining metrics in Table \ref{tab:mfam_w2_comparison}, together with the bottom row of Figure \ref{fig:2D_traj}, show that path supervision substantially improves the agreement between the generated particle distribution and the prescribed intermediate marginals. This demonstrates that intermediate marginal supervision effectively controls the entire distributional trajectory rather than only its terminal endpoint, albeit with some loss of terminal matching accuracy in this example.
\begin{figure}[h]
    \centering 
    \includegraphics[width=1.0\textwidth]{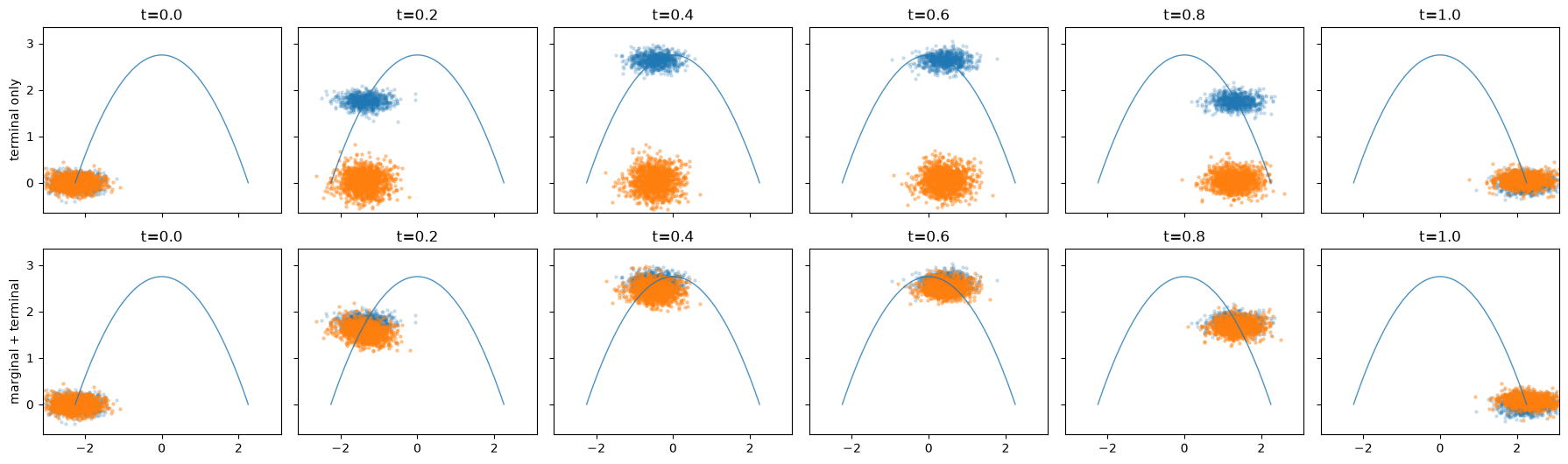} 
    \caption{MFC control with path supervision, a 2-D Gaussian example with given time deterministic trajectory for mean. Top row: the generated Gaussian random variables without path supervision. Bottom row: the generated Gaussian samples with constraints over path.}
    \label{fig:2D_traj} 
\end{figure}
\begin{table}[htbp]
\centering
\caption{Wasserstein distance comparison of the two MFAM approaches. 
Results are reported as mean $\pm$ standard error over 10 independent runs. The $W_2$ is calculated also over $t=0.1,0.2,...,0.9$ total 9 observation times.}
\label{tab:mfam_w2_comparison}
\begin{tabular}{lcccc}
\toprule
Model 
& Terminal $W_2$
& Max intermediate $W_2$
& Mean path $W_2$
& Mean observed $W_2$ \\
\midrule
Terminal-only MFAM
& $0.030 \pm 0.004$
& $2.750 \pm 0.006$
& $1.832 \pm 0.005$
& $2.015 \pm 0.005$ \\

Marginally supervised MFAM
& $0.129 \pm 0.003$
& $0.193 \pm 0.004$
& $0.124 \pm 0.003$
& $0.127 \pm 0.003$ \\
\bottomrule
\end{tabular}
\end{table}

\subsection{Example 5: Smile transport using celebrity data}
In this section, we consider an unpaired facial-expression transport task in which the distribution of non-smiling faces is transported toward that of smiling faces. Both non-smiling and smiling images are first encoded into the latent space of a frozen variational autoencoder (VAE), and the proposed algorithm learns a time-dependent feedback field $u(t,x)$ through the controlled latent dynamics
\begin{align}
    dX_t=u(t,X_t)dt+\sigma dW_t.
\end{align}
The purpose of this experiment is primarily to demonstrate the applicability and numerical feasibility of Algorithm \ref{algorithm_main2} in a high-dimensional, real-world generative setting.

We use the pretrained VAE checkpoint \texttt{stabilityai/sd-vae-ft-mse}, a Stable Diffusion VAE available through Hugging Face \texttt{diffusers}, for image encoding and decoding. CelebA images are center-cropped to $160\times160$ pixels and resized to $128\times128$. Before being passed to the VAE, the images are rescaled to $[-1,1]$. Each image is encoded using the posterior mean of the VAE encoder. The resulting latent representation has dimension $4\times16\times16$ and is flattened into a $1024$-dimensional vector.

We construct training cohorts containing $10{,}000$ smiling and $10{,}000$ non-smiling faces, together with additional cohorts of $2{,}000$ smiling and $2{,}000$ non-smiling faces. The two $10{,}000$-sample training cohorts are pooled to fit a coordinatewise standardization transformation, which is subsequently applied to all latent representations. They are also used to train an auxiliary latent-space smile classifier for diagnostic purposes. For the proof-of-concept MFC experiment reported here, the empirical target distribution used in the Sinkhorn terminal loss is formed from the $2{,}000$ smiling latent samples, while source particles are sampled from the $10{,}000$ non-smiling training cohort. Accordingly, the distributional metrics reported below measure how accurately the controlled flow matches this prescribed empirical target distribution; they are intended as diagnostics of the transport procedure rather than as measures of out-of-sample generalization.

The control objective is
\begin{align}
    J(u)=\frac12\E\left[\int_0^T|u(t,X_t)|^2dt+\gamma|X_T-X_0|^2\right]+\lambda_gD_{\rm Sink}(\mu_T,\rho_{\rm target}),
\end{align}
where the source-preservation term $\frac{\gamma}{2}\E[|X_T-X_0|^2]$ is introduced to discourage excessive displacement from the initial latent representation and thereby preserve image characteristics unrelated to the desired expression change. This path-dependent terminal term can be incorporated into the preceding framework by augmenting the state with the static component $X_0$. We take $\gamma=0.10$ and $\lambda_g=900$.

For the generated and target empirical measures $  \widehat\mu_T:=\frac1M\sum_{i=1}^M\delta_{x_i},
    \widehat\rho_{\rm target}:=\frac1H\sum_{j=1}^H\delta_{y_j},$
the terminal distributional discrepancy is measured using the Sinkhorn divergence
\begin{align}
    S_\epsilon(\widehat\mu_T,\widehat\rho_{\rm target})
    :=\operatorname{OT}_\epsilon(\widehat\mu_T,\widehat\rho_{\rm target})
    -\frac12\operatorname{OT}_\epsilon(\widehat\mu_T,\widehat\mu_T)
    -\frac12\operatorname{OT}_\epsilon(\widehat\rho_{\rm target},\widehat\rho_{\rm target}).
\end{align}
The scalar Sinkhorn loss is differentiated with respect to each generated terminal particle using automatic differentiation. Under the squared-distance convention used in the implementation, the empirical first-variation field is approximated by
\begin{align}
    h_{\rm Sink}^i=2M\nabla_{x_i}S_\epsilon(\widehat\mu_T,\widehat\rho_{\rm target}), \label{eq:sinkhorn_force}
\end{align}
and the terminal adjoint contribution is therefore
\begin{align}
    Y_N^i=\lambda_gh_{\rm Sink}^i+\gamma(X_T^i-X_0^i).
\end{align}

We use $N=80$ time steps with diffusion coefficient $\sigma=0.15$ and $M=256$ generated particles in each outer iteration. At each iteration, the Sinkhorn force \eqref{eq:sinkhorn_force} is averaged over four independently sampled target batches of size $H=512$. The feedback control is trained for $1200$ outer iterations using the projected functional-descent update with base step size $0.015$ and an exponentially decaying trust-region radius.

For this proof-of-concept experiment, the sliced $W_2$ distance between the generated terminal distribution and the prescribed smiling target distribution decreases from $0.216$ before transport to $0.0969$ after training, corresponding to a reduction of approximately $55.1\%$. As shown in Figure \ref{fig:smile_drift_only}, the transported images acquire a visibly smiling expression while largely preserving other characteristics of the source image, including the background, hairstyle, clothing, and overall facial appearance. These results provide evidence that the proposed sample-based mean-field control procedure can be implemented effectively in a $1024$-dimensional latent space and can produce meaningful semantic transport on real image data.

\begin{figure}[h]
    \centering 
    \includegraphics[width=1.0\textwidth]{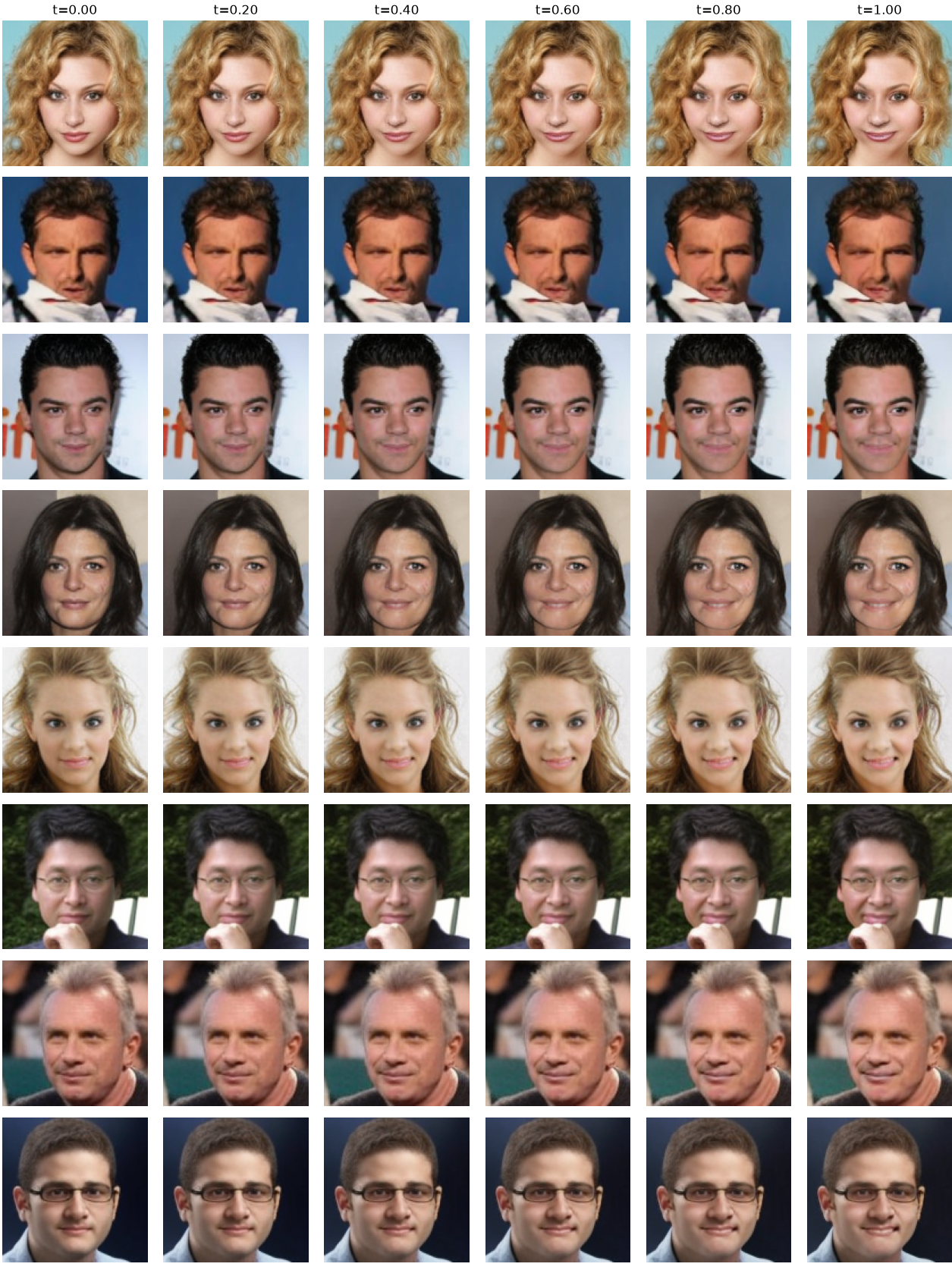} 
    \caption{Facial transport after $1200$ training iterations with the best checkpoint result. Left-to-right: the non-smiling faces are gradually transported to smiling faces.}
    \label{fig:smile_drift_only} 
\end{figure}

\bibliographystyle{apacite}

\begin{thebibliography}{99}

\bibitem{Anderson}
D.~Andersson and B.~Djehiche.
\newblock A maximum principle for SDEs of mean-field type.
\newblock \emph{Appl. Math. Optim.}, 63:341--356, 2011.

\bibitem{Hui1}
R.~Archibald, F.~Bao, Y.~Cao, and H.~Sun.
\newblock Numerical analysis for convergence of a sample-wise backpropagation
  method for training stochastic neural networks.
\newblock \emph{SIAM J. Numer. Anal.}, 62(2):593--621, 2024.

\bibitem{Hui3}
F.~Bao and H.~Sun.
\newblock Batch sample-wise stochastic optimal control via stochastic maximum
  principle.
\newblock \emph{arXiv preprint arXiv:2505.02688}, 2025.

\bibitem{feng1}
R.~Archibald, F.~Bao, Y.~Cao, and H.~Zhang.
\newblock A backward SDE method for uncertainty quantification in deep
  learning.
\newblock \emph{Discrete Contin. Dyn. Syst. Ser. S}, 15(7):2807--2835, 2022.

\bibitem{Cai}
W.~Cai, S.~Fang, and T.~Zhou.
\newblock SOC-MartNet: A martingale neural network for the
  Hamilton--Jacobi--Bellman equation without explicit
  $\inf\nolimits_{u\in U}h$ in stochastic optimal controls.
\newblock \emph{SIAM J. Sci. Comput.}, 47(4):795--819, 2025.

\bibitem{Bensoussan1_T}
A.~Bensoussan.
\newblock Lecture on stochastic control.
\newblock In \emph{Nonlinear Filtering and Stochastic Control}, volume 972 of
  \emph{Lecture Notes in Mathematics}, pages 1--62.
\newblock Springer-Verlag, Berlin, New York, 1982.

\bibitem{Oksendal1}
F.~Biagini, Y.~Hu, B.~\O ksendal, and A.~Sulem.
\newblock A stochastic maximum principle for processes driven by fractional
  Brownian motion.
\newblock \emph{Stochastic Process. Appl.}, 100(1--2):233--253, 2002.

\bibitem{Buckdahn1}
R.~Buckdahn, J.~Li, S.~Peng, and C.~Rainer.
\newblock Mean-field stochastic differential equations and associated PDEs.
\newblock \emph{Ann. Probab.}, 45(2):824--878, 2017.

\bibitem{Carmona0}
R.~Carmona.
\newblock \emph{Lectures on BSDEs, Stochastic Control, and Stochastic
  Differential Games with Financial Applications}.
\newblock SIAM, Philadelphia, PA, 2016.

\bibitem{carmona1}
R.~Carmona, J.~P.~Fouque, and L.~Sun.
\newblock Mean field games and systemic risk.
\newblock \emph{Commun. Math. Sci.}, 13(4):911--933, 2015.

\bibitem{Carmona2}
R.~Carmona and M.~Lauri\`ere.
\newblock Convergence analysis of machine learning algorithms for the numerical
  solution of mean field control and games I: The ergodic case.
\newblock \emph{SIAM J. Numer. Anal.}, 59(3):1455--1485, 2021.

\bibitem{Carmona3}
R.~Carmona and M.~Lauri\`ere.
\newblock Convergence analysis of machine learning algorithms for the numerical
  solution of mean field control and games II: The finite horizon case.
\newblock \emph{Ann. Appl. Probab.}, 32(6):4065--4105, 2022.

\bibitem{Carmona4}
R.~Carmona and M.~Lauri\`ere.
\newblock Deep learning for mean field games and mean field control with
  applications to finance.
\newblock In J.~J.~Hasbrouck and T.~J.~Sargent, editors,
  \emph{Deep Learning in Economics}, pages 369--392.
\newblock Cambridge University Press, 2023.

\bibitem{Carmona5}
B.~Acciaio, J.~Backhoff-Veraguas, and R.~Carmona.
\newblock Extended mean field control problems: Stochastic maximum principle
  and transport perspective.
\newblock \emph{SIAM J. Control Optim.}, 57(6):3666--3693, 2019.

\bibitem{jiequn2}
C.~Domingo-Enrich, J.~Han, B.~Amos, J.~Bruna, and R.~T.~Q.~Chen.
\newblock Stochastic optimal control matching.
\newblock \emph{arXiv preprint arXiv:2312.02027}, 2023.

\bibitem{Du1}
N.~Du, J.~T.~Shi, and W.~B.~Liu.
\newblock An effective gradient projection method for stochastic optimal
  control.
\newblock \emph{Int. J. Numer. Anal. Model.}, 4(4):757--774, 2013.

\bibitem{Weinan2}
W.~E, J.~Han, and A.~Jentzen.
\newblock Deep learning-based numerical methods for high-dimensional parabolic
  partial differential equations and backward stochastic differential
  equations.
\newblock \emph{Commun. Math. Stat.}, 5(4):349--380, 2017.

\bibitem{HanJentzen1}
J.~Han, A.~Jentzen, and W.~E.
\newblock Solving high-dimensional partial differential equations using deep
  learning.
\newblock \emph{Proc. Natl. Acad. Sci. U.S.A.}, 115(34):8505--8510, 2018.

\bibitem{Zhao1}
B.~Gong, W.~Liu, T.~Tang, W.~Zhao, and T.~Zhou.
\newblock An efficient gradient projection method for stochastic optimal
  control problems.
\newblock \emph{SIAM J. Numer. Anal.}, 55(6):2982--3005, 2017.

\bibitem{QHan}
Q.~Han and S.~Ji.
\newblock A multi-step algorithm for BSDEs based on a predictor-corrector
  scheme and least-squares Monte Carlo.
\newblock \emph{Methodol. Comput. Appl. Probab.}, 24(4):2403--2426, 2022.

\bibitem{jiequn1}
J.~Han and W.~E.
\newblock Deep learning approximation for stochastic control problems.
\newblock In \emph{Advances in Neural Information Processing Systems, Deep
  Reinforcement Learning Workshop}, 2016.

\bibitem{lauriere0}
M.~Han, M.~Lauri\`ere, and E.~Vanden-Eijnden.
\newblock A simulation-free deep learning approach to stochastic optimal
  control.
\newblock \emph{arXiv preprint arXiv:2410.05163}, 2024.

\bibitem{Hanson}
F.~B.~Hanson.
\newblock \emph{Applied Stochastic Processes and Control for Jump-Diffusions:
  Modeling, Analysis, and Computation}.
\newblock SIAM, Philadelphia, PA, 2007.

\bibitem{Haussmann}
U.~G.~Haussmann.
\newblock Some examples of optimal stochastic controls or: The stochastic
  maximum principle at work.
\newblock \emph{SIAM Rev.}, 23(2):292--307, 1981.

\bibitem{Kushner}
H.~J.~Kushner.
\newblock Numerical methods for stochastic control problems in continuous time.
\newblock \emph{SIAM J. Control Optim.}, 28(5):999--1026, 1990.

\bibitem{Lars}
X.~Li, D.~Verma, and L.~Ruthotto.
\newblock A neural network approach for stochastic optimal control.
\newblock \emph{SIAM J. Sci. Comput.}, 46(5):535--556, 2024.

\bibitem{Qli}
Q.~Li, L.~Chen, C.~Tai, and W.~E.
\newblock Maximum principle based algorithms for deep learning.
\newblock \emph{J. Mach. Learn. Res.}, 18(1):5998--6026, 2018.

\bibitem{Ruimeng1}
M.~Min and R.~Hu.
\newblock Signatured deep fictitious play for mean field games with common
  noise.
\newblock In \emph{Proceedings of the 38th International Conference on Machine
  Learning}, volume 139 of \emph{Proceedings of Machine Learning Research},
  pages 7731--7740.
\newblock PMLR, 2021.

\bibitem{Peng1}
S.~Peng.
\newblock Backward stochastic differential equations and applications to
  optimal control.
\newblock \emph{Appl. Math. Optim.}, 27(2):125--144, 1993.

\bibitem{Peng4}
S.~Peng.
\newblock A general stochastic maximum principle for optimal control problems.
\newblock \emph{SIAM J. Control Optim.}, 28(4):966--979, 1990.

\bibitem{Peng5}
S.~Peng and E.~Pardoux.
\newblock Backward stochastic differential equations and quasilinear parabolic
  partial differential equations.
\newblock In B.~L.~Rozovskii and R.~B.~Sowers, editors,
  \emph{Stochastic Partial Differential Equations and Their Applications},
  volume 176 of \emph{Lecture Notes in Control and Information Sciences},
  pages 200--217.
\newblock Springer, Berlin, Heidelberg, 1992.

\bibitem{pham1}
H.~Pham.
\newblock \emph{Continuous-Time Stochastic Control and Optimization with
  Financial Applications}, volume 61 of \emph{Stochastic Modelling and Applied
  Probability}.
\newblock Springer, Berlin, 2009.

\bibitem{pham2}
H.~Pham and X.~Warin.
\newblock Mean-field neural networks-based algorithms for McKean--Vlasov
  control problems.
\newblock \emph{J. Mach. Learn. Model. Comput.}, 3(2):176--214, 2024.

\bibitem{pham3}
H.~Pham and X.~Warin.
\newblock Actor-critic learning algorithms for mean-field control with moment
  neural networks.
\newblock \emph{arXiv preprint arXiv:2309.04317}, 2023.

\bibitem{pham4}
H.~Pham and X.~Wei.
\newblock Bellman equation and viscosity solutions for mean-field stochastic
  control problem.
\newblock \emph{ESAIM Control Optim. Calc. Var.}, 24(1):437--461, 2018.

\bibitem{pham5}
S.~El~Boustany, S.~Mekkaoui, Y.~Hafsi, A.~Alouadi, and H.~Pham.
\newblock Learning generative dynamics with soft law constraints:
  A McKean--Vlasov FBSDE approach.
\newblock \emph{arXiv preprint arXiv:2605.08928}, 2026.

\bibitem{Jiang1}
Y.~Jiang, R.~Xu, and L.~Zhang.
\newblock Schr{\"o}dinger bridge with transport relaxation.
\newblock \emph{arXiv preprint arXiv:2602.08118}, 2026.

\bibitem{Hui2}
H.~Sun.
\newblock Meshfree approximation for stochastic optimal control problems.
\newblock \emph{Commun. Math. Res.}, 37(3):387--420, 2021.

\bibitem{Teichmann}
H.~M.~Soner, J.~Teichmann, and Q.~Yan.
\newblock Learning algorithms for mean field optimal control.
\newblock \emph{arXiv preprint arXiv:2503.17869}, 2025.

\bibitem{Teichmann1}
C.~Herrera, F.~Krach, P.~Ruyssen, and J.~Teichmann.
\newblock Optimal stopping via randomized neural networks.
\newblock \emph{Front. Math. Finance}, 3(1):31--77, 2025.

\bibitem{Yong1}
J.~Yong and X.~Y.~Zhou.
\newblock \emph{Stochastic Controls: Hamiltonian Systems and HJB Equations},
  volume 43 of \emph{Applications of Mathematics}.
\newblock Springer, New York, 1999.

\bibitem{Jianfeng}
J.~Zhang.
\newblock \emph{Backward Stochastic Differential Equations: From Linear to
  Fully Nonlinear Theory}, volume 86 of \emph{Probability Theory and
  Stochastic Modelling}.
\newblock Springer, 2017.

\bibitem{ZhangIEEE}
R.~Zhang, Y.~Lan, G.-B.~Huang, and Z.-B.~Xu.
\newblock Universal approximation of extreme learning machine with adaptive
  growth of hidden nodes.
\newblock \emph{IEEE Trans. Neural Netw. Learn. Syst.}, 23(2):365--371, 2012.

\bibitem{Peng3}
W.~Zhao, L.~Chen, and S.~Peng.
\newblock A new kind of accurate numerical method for backward stochastic
  differential equations.
\newblock \emph{SIAM J. Sci. Comput.}, 28(4):1563--1581, 2006.

\bibitem{tsitsiklis2001regression}
J.~N.~Tsitsiklis and B.~Van~Roy.
\newblock Regression methods for pricing complex American-style options.
\newblock \emph{IEEE Trans. Neural Netw.}, 12(4):694--703, 2001.

\bibitem{carmona2015probabilistic}
R.~Carmona and D.~Lacker.
\newblock A probabilistic weak formulation of mean field games and
  applications.
\newblock \emph{Ann. Appl. Probab.}, 25(3):1189--1231, 2015.

\bibitem{carmona2018probabilistic1}
R.~Carmona and F.~Delarue.
\newblock \emph{Probabilistic Theory of Mean Field Games with Applications I},
  volume 83 of \emph{Probability Theory and Stochastic Modelling}.
\newblock Springer, Cham, 2018.

\bibitem{cardaliaguet2012}
P.~Cardaliaguet.
\newblock Notes from P.-L. Lions' lectures at the Coll\`ege de France.
\newblock Technical report, 2012.

\bibitem{DomingoEnrich1}
C.~Domingo-Enrich, M.~Drozdzal, B.~Karrer, and R.~T.~Q.~Chen.
\newblock Adjoint matching: Fine-tuning flow and diffusion generative models
  with memoryless stochastic optimal control.
\newblock In \emph{The Thirteenth International Conference on Learning
  Representations}, 2025.

\bibitem{MaTanXu1}
J.~Ma, Y.~Tan, and R.~Xu.
\newblock Schr{\"o}dinger bridge for generative AI: Soft-constrained
  formulation and convergence analysis.
\newblock \emph{arXiv preprint arXiv:2510.11829}, 2025.

\end{thebibliography}

\end{document}